%% file: intro.tex
\documentclass[a4paper,10pt]{amsart}
\usepackage{latexsym}
\usepackage{amssymb}
\usepackage{amsmath}

\input{lamac}
\input{diagram}

\begin{document}
\date{}
\title{Linear Algebra Over a Ring}
\author{Ivo Herzog}
\address{Yonsei University, Sinchon-dong, Seodaemun-gu, Seoul 120-749, South Korea}
\email{herzog@yonsei.ac.kr}
\address{The Ohio State University at Lima, Lima, OH 45804 USA}
\email{herzog.23@osu.edu}
\thanks{The author was partially supported by a grant from Yonsei University and NSF Grant DMS-05-01207.}

\begin{abstract} 
Given an associative, not necessarily commutative, ring $R$ with identity, a formal matrix calculus is introduced and developed for pairs of matrices over $R.$ This calculus subsumes the theory of homogeneous systems of linear equations with coefficients in $R.$ In the case when the ring $R$ is a field, every pair is equivalent to a homogeneous system. 

Using the formal matrix calculus, two alternate presentations are given for the \linebreak Grothendieck group 
$K_0 (\Rmod, \oplus)$ of the category $\Rmod$ of finitely presented modules.  One of these presentations suggests a homological interpretation, and so a complex is introduced whose $0$-dimensional homology is naturally isomorphic to $K_0 (\Rmod, \oplus).$ A computation shows that if $R = k$ is a field, then the $1$-dimensional homology group is given by $(k^{\times})_{\ab}/\{ \pm 1 \},$ where $k^{\times}$ denotes the multiplicitave group of $k,$ and $(k^{\times})_{\ab}$ its abelianization.

The formal matrix calculus, which consists of three rules of matrix operation, is the syntax of a deductive system whose completeness was proved by Prest. The three rules of inference of this deductive system correspond to the three rules of matrix operation, which appear in the formal matrix calculus as the Rules of Divisibility.
\end{abstract}

\maketitle

\footskip30pt

\pagestyle{plain}

Let $R$ be an associative, not necessarily commutative, ring with identity. For every natural number $n,$ define 
$L'_n (R)$ to be the collection of pairs $(B \; |\; A)$ where $A$ and $B$ are matrices with entries in $R$ such that $A$ has $n$ columns, and $B$ has the same number of rows as $A.$ Define the relation $$(B \; |\; A) \leq_n (B' \; |\; A')$$ to hold in $L'_n (R)$ provided there exist matrices $U,$ $V$ and $G,$ of appropriate size, such that $$UB = B'V \mbox{  and  } UA = A' + B'G.$$
Separating out the individual roles of the three matrices, one verifies easily (Theorem~\ref{rules}) that this relation is the least pre-order on $L'_n (R)$ satisfying the following three Rules of (Left) Divisibility (RoD) for matrices: 
\begin{enumerate}
\item if $U$ is a matrix with $m$ columns, then $(B \; | \; A) \leq_n (UB \; | \; UA).$ 
\item if $V$ is a matrix with $k$ rows, then $(BV \; | \; A) \leq_n (B \; | \; A).$
\item if $G$ is a $k \times n$ matrix, then $(B \; | \; A + BG) \leq_n (B \; | \; A)$
\end{enumerate}
In this article, we develop the formal matrix calculus that arises from these rules.

Two pairs in $L'_n (R)$ are equivalent $(B \; |\; A) \approx_n (B' \; |\; A')$ provided both
$$(B \; |\; A) \leq_n (B' \; |\; A') \mbox{ and } (B' \; |\; A') \leq_n (B \; |\; A)$$
hold; an $n${\em -ary matrix pair} $[B \; | \; A]$ is defined to be an equivalence class of this relation. The collection of $n$-ary matrix pairs is denoted by $L_n (R);$ the partial order induced on $L_n (R)$ by $\leq_n$ is denoted using the same notation. This partial order $L_n (R)$ of $n$-ary matrix pairs has a maximum element $1_n$ that satisfies \bigskip

\noindent {\bf Proposition~\ref{maximum}.} $[B \; | \; A] = 1_n$ if and only if there exists a matrix $W$ such that $A = BW.$ \bigskip

\noindent The element $(B \; | \; A)$ in $L'_n (R)$ may thus be interpreted as the proposition ''$B$ divides $A$ on the left'' and the $n$-ary matrix pair $[B \; | \; A]$ as a point in the partial order $L_n (R)$ that measures the extent to which $B$ divides $A$ on the left. Seen in this light, RoD (1), for example, asserts that the likelihood that $UB$ divides $UA$ on the left is at least as high as the likelihood that $B$ divides $A$ on the left. 

An element of $L'_n (R)$ is called a ({\em homogeneous}) {\em system} ({\em of linear equations in $n$ variables}) provided it is of the form $(\zero \; | \; A).$ The {\em matrix of coefficients} of this system is $A$ and, for simplicity, we write $(A) := (\zero \; | \; A),$ and denote the associated $n$-ary matrix pair by $[A] := [\zero \; | \; A].$ It follows from the definition that if $A$ and $A'$ are matrices, both with $n$ columns, then $(A) \approx_n (A')$ in $L'_n (R)$ if and only if there exist matrices $U$ and $U'$ such that $UA = A'$ and $U'A' = A.$ This implies that for every left $R$-module ${_R}M,$ the two solution subgroups of $({_R}M)^n$ of the homogeneous systems with matrices of coefficients $A$ and $A',$ respectively, are the same.\bigskip 

\noindent {\bf Corollary~\ref{von Neumann}.} The ring $R$ is von Neumann regular if and only if for every $n \geq 1$ (resp., $n = 1),$ every $(B \; | \; A) \in L'_n (R)$ is equivalent to a homogeneous system.\bigskip

\noindent A field $k$ is certainly a von Neumann regular ring, so the formal matrix calculus on $L_n (k)$ coincides with the study of homogeneous systems of linear equations in $n$ variables, with coefficients in $k.$ The other main features of this matrix calculus may be summarized as follows. \bigskip

\noindent {\bf Corollary~\ref{lattice duality}.} The map 
${\dis [B \; | \; A] \mapsto \left[ \begin{array}{c|c}	B^{\tr} & \zero \\ A^{\tr} & I_n \end{array} 	\right]}$ 
is an anti-isomorphism between $L_n (R)$ and $L_n (R^{\op}).$\bigskip

\noindent {\bf Theorem~\ref{modular}.} The partial order $L_n (R)$ is a modular lattice with maximum and minimum elements.\bigskip

\noindent {\bf Theorem~\ref{epi}.} A morphism $f : R \to S$ of rings is an epimorphism if and only if for every $n \geq 1$ (resp. $n = 2$), the induced morphism $L_n (f) : L_n (R) \to L_n (S)$ is onto. \bigskip

The elements $(B \; | \; A)$ of $L'_n (R)$ may be interpreted as the syntax of a deductive system, whose completeness was proved by Prest (Lemma~1.1.13 of~\cite{Mike's book}. The three matrices that appear in the statement of Prest's result correspond to the three rules of inference (cf.\ the commentary preceding {\em ibid.}) of the deductive system, and in the formal matrix calculus find expression as the Rules of Divisibility (RoD) (1)-(3). Precisely, let $\Lang (R)$ be the language for left $R$-modules, and denote by $T(R)$ the standard collection of axioms, expressible in $\Lang (R),$ for a left $R$-module. Associate to the element $(B \; | \; A)$ of $L'_n (R)$ the formula
$$(B \; | \; A)(\bv ) :=  \exists \bw \; (B\bw \doteq A \bv),$$
where $\bv$ is a column $n$-vector of variables $v_i$ and $\bw$ a {\em column} $k$-vector of variables $w_j.$ Then we may introduce a pre-order order on $L'_n (R)$ by defining $(B \; | \; A) \vdash_n (B' \; | \; A')$ to hold provided that
$$T(R) \vdash \forall \bv \; (B|A\bv \to B'|A'\bv).$$
It is readily verified that this pre-order obeys the three Rules of Divisibility, and since $\leq_n$ is the least pre-order that obeys these rules, $(B \; | \; A) \leq_n (B' \; | \; A')$ implies
$(B \; | \; A) \vdash_n (B' \; | \; A').$
\bigskip

\noindent {\bf Theorem~\ref{Lemma Presta 1}} ({\bf Lemma Presta I.})~\cite[Lemma 1.1.13 and Cor.\ 1.1.16]{Mike's book} 
Given $(B \; | \; A)$ and $(B' \; | \; A')$ in $L'_n (R),$ 
$$(B \; | \; A) \leq_n (B' \; | \; A') \;\;\; \mbox{if and only if} \;\;\; (B \; | \; A) \vdash_n (B' \; | \; A').$$
\bigskip

\noindent Lemma Presta is a completeness theorem, for it shows that any implication between formulae of the 
form $(B \; | \; A)(\bv )$ that is provable relative to the axioms $T(R)$ is provable using the three Rules of Divisibility, construed as rules of inference. Once the relationship between these two partial orders 
on $L'_n (R)$ is established, the results cited above (Corollaries~\ref{von Neumann} and~\ref{lattice duality} and Theorems~\ref{modular} and~\ref{epi}) appear as familiar results from the model theory of modules. Corollary~\ref{von Neumann} is just elemination of quantifiers~\cite[Thm.\ 2.3.24]{Mike's book}; Corollary~\ref{lattice duality} the anti-isomorphism, discovered by Prest~\cite[\S 1.3.1]{Mike's book}, and independently, by Huisgen-Zimmermann and Zimmermann~\cite{ZZH}, between the respective partial orders of positive-primitive formulae over $R$ and the opposite ring $R^{\op};$ Theorem~\ref{modular} asserts nothing more than the fact that these partial orders are modular lattices 
(cf.~\cite[\S 1.1.3]{Mike's book}); and Theorem~\ref{epi} is a variation of Prest's result~\cite[Thm.\ 6.1.8]{Mike's book} that if $f : R \to S$ is a ring epimorphism, then the subcategory $\SMod \subseteq \RMod$ is axiomatizable in the language $\Lang (R).$  

Let $A$ be an $m \times n$ matrix; $B$ and $m \times k$ matrix and ${_R}M$ a left $R$-module. Let us consider the {\em nonhomogeneous system} of linear equations $$A \bv \doteq \bb,$$
where $\bv$ is a column $n$-vector of variables $(v_i)$ and $\bb$ a column $k$-vector with entries from $M.$ Denote by
$$\Sol_M (A \bv \doteq \bb) := \{ \ba \in ({_R}M)^n \; : \; A\ba = \bb \}$$
the subgroup of $({_R}M)^n$ of solutions in $M$ to the nonhomogeneous system. Given an element 
$(B \; | \; A) \in L'_n (R),$ define 
$$(B \; | \; A)({_R}M) := \bigcup_{\bb \in BM^k} \Sol_M (A \bv \doteq \bb).$$
This is consistent with model-theoretic notation, because $(B \; | \; A)({_R}M)$ is the subgroup of $({_R}M)^n$
defined in ${_R}M$ by the formula $(B \; | \; A)(\bv).$ Let us introduce the relation 
$$(B \; | \; A) \models_n (B' \; | \; A')$$ to hold in $L'_n (R)$ provided that 
$(B \; | \; A)({_R}M) \subseteq (B' \; | \; A')({_R}M)$ for every left $R$-module ${_R}M.$ Equivalently, 
$$T(R) \models \forall \bv \; (B|A\bv \to B'|A'\bv).$$
By G\"{o}del's Completeness Theorem, the relations $\vdash_n$ and $\models_n$ are the same on $L'_n (R),$ so that one 
obtains a second version of Lemma Presta.\bigskip

\noindent {\bf Proposition~\ref{Lemma Presta 2}} ({\bf Lemma Presta II.}) Given $(B \; | \; A)$ and $(B' \; | \; A')$ in $L'_n (R),$ 
$$(B \; | \; A) \leq_n (B' \; | \; A') \;\;\; \mbox{if and only if} \;\;\; (B \; | \; A) \models_n (B' \; | \; A').$$
\bigskip

Seen in this light, the Rules of Divisibility declare the relationships between spaces of nonhomogeneous systems of linear equations. For example, RoD (2) may be seen as the statement that for every left $R$-module ${_R}M,$  
$$\bigcup_{\bb \in BV \cdot M^{k'}} \Sol_M (A\bv \doteq \bb) \subseteq 
\bigcup_{\bc \in B \cdot M^{k}} \Sol_M (A\bv \doteq \bc).$$

Suppose that $A$ and $A'$ are matrices, both with $n$ columns, with the property that for every left $R$-module ${_R}M,$ the two solution subgroups of $({_R}M)^n$ of the two respective homogeneous systems of linear equations are equal. This is expressible in $\Lang (R)$ by 
$$T(R) \models \forall \bv \; [(A\bv \doteq \zero) \leftrightarrow (A'\bv \doteq \zero)].$$
Equivalently, $(A) \models_n (A')$ and $(A') \models_n (A).$ By Lemma Presta II, the two systems $(A)$ and $(A')$ are equivalent in $L'_n (R).$\bigskip

The first section of the article is devoted to a general exposition of the formal matrix calculus; the second section of the article applies the formal matrix calculus to obtain various presentations of the Grothendieck group 
$K_0 (\Rmod, \oplus)$ of the category $\Rmod$ of finitely presented left $R$-modules. One of these presentations suggests a homological interpretation, and so a complex $C_* (R)$ is introduced in the third section whose $0$-dimensional homology $H_0 (R)$ is naturally isomorphic (Theorem~\ref{0 homo}) to the Grothendieck group 
$K_0 (\Rmod, \oplus).$ If $R = k$ is a field, then the $1$-dimensional homology $H_1 (k)$ is isomorphic (Corollary~\ref{1-dim homo}) to $(k^{\times})_{\ab}/\{ \pm 1 \},$ where $(k^{\times})_{\ab}$ denotes the abelianization of the multiplicitave group of $k.$ The last section describes some representations of the partial orders $L_n (R).$ These representation offer concrete examples, mostly coming from the model theory of models, of isomorphic partial orders, which also serve as a historical reference for the formal matrix calculus they inspire. 

Throughout the article, $R$ will denote an associative ring with identity $1.$ The $n \times n$ identity matrix will be denoted by $I_n.$ At first (in most of section \S 1.1) the dimensions of zero matrices are specified; the $m \times n$ zero matrix is denoted by ${_m}\zero_n,$ but then these cumbersome subscripts are dropped to denote any zero matrix by $\zero,$ except the $1 \times 1,$ which is denoted $0.$ If $A$ is an $m \times n$ matrix, then ${_i}A$ will denote the $i$-th row of $A$ and $A_j$ the $j$-th column. When useful, the matrix $A$ is expressed as a row of column matrices
$A = (A_1, A_2, \ldots, A_n).$ If the matrix $A$ is $m \times (n + 1),$ the columns are indexed so that
$A = (A_0, A_1, \ldots, A_{n+1}).$

A great debt is owing to Puninsky, who showed me how to manipulate matrices using Lemma Presta. The formal matrix calculus presented here was first introduced at a lecture in the Durham Symposium, ''New Directions in the Model Theory of Fields,'' July 2009. I am grateful to M.\ Makkai and Ph.\ Rothmaler for their constructive feedback. 

\input{fc}
\input{group}

\input{homo}

\input{appendix}

\end{document}

%% file: lamac.tex
\newtheorem{theorem}{Theorem}
\newtheorem{proposition}[theorem]{Proposition}
\newtheorem{corollary}[theorem]{Corollary}
\newtheorem{lemma}[theorem]{Lemma}
\newtheorem{example}[theorem]{Example}

\newcommand{\RMod}{\bf R{\rm \mbox{{\bf -Mod}}}}
\newcommand{\SMod}{\bf S{\rm \mbox{{\bf -Mod}}}}
\newcommand{\Rmod}{R{\rm \mbox{-mod}}}
\newcommand{\modR}{{\rm \mbox{mod-}}R}
\newcommand{\Smod}{S{\rm \mbox{-mod}}}

\newcommand{\Ring}{\bf {\rm \mbox{{\bf Ring}}}}
\newcommand{\bAb}{\bf {\rm \mbox{{\bf Ab}}}}

\newcommand{\Gal}{{\rm \mbox{Gal}}}

\newcommand{\op}{{\rm \mbox{op}}}
\newcommand{\tr}{{\rm \mbox{tr}}}
\newcommand{\ab}{{\rm \mbox{ab}}}
\newcommand{\Ab}{{\rm \mbox{Ab}}}

\newcommand{\Ev}{{\rm \mbox{Ev}}}

\newcommand{\Sol}{{\rm \mbox{Sol}}}
\newcommand{\Latt}{{\rm \mbox{Latt}}}

\newcommand{\Coker}{{\rm \mbox{Coker }}}

\newcommand{\ba}{ {\bf a}}
\newcommand{\bb}{ {\bf b}}
\newcommand{\bc}{ {\bf c}}
\newcommand{\bd}{ {\bf d}}

\newcommand{\bv}{ {\bf v}}
\newcommand{\bw}{ {\bf w}}

\newcommand{\zero}{ {\bf 0}}
\newcommand{\bei}{ {\bf e_i}}

\newcommand{\grg}{\gamma}
\newcommand{\gri}{\iota}

\newcommand{\grl}{\lambda}

\newcommand{\grd}{\delta}

\newcommand{\grG}{\Gamma}

\newcommand{\grk}{\kappa}

\newcommand{\grf}{\varphi}

\newcommand{\bZ}{\mathbb{bZ}}
\newcommand{\isom}{\cong}
\newcommand{\tensor}{\otimes}

\newcommand{\Lang}{{\mathcal L}}

\newcommand{\dis}{\displaystyle}

%% file: diagram.tex
\newcommand{\tlowername}[2]%
{$\stackrel{\makebox[1pt]{#1}}%
{\begin{picture}(0,0)%
\put(0,0){\makebox(0,6)[t]{\makebox[1pt]{$#2$}}}%
\end{picture}}$}%

\newcommand{\AR}[1]%
{\begin{picture}(#1,0)%
\put(0,0){\vector(1,0){#1}}%
\end{picture}}%

\newcommand{\DOTAR}[1]%
{\NUMBEROFDOTS=#1%
\divide\NUMBEROFDOTS by 3%
\begin{picture}(#1,0)%
\multiput(0,0)(3,0){\NUMBEROFDOTS}{\circle*{1}}%
\put(#1,0){\vector(1,0){0}}%
\end{picture}}%

\newcommand{\MONO}[1]%
{\begin{picture}(#1,0)%
\put(0,0){\vector(1,0){#1}}%
\put(2,-2){\line(0,1){4}}%
\end{picture}}%

\newcommand{\EPI}[1]%
{\begin{picture}(#1,0)(-#1,0)%
\put(-#1,0){\vector(1,0){#1}}%
\put(-6,-2){\line(0,1){4}}%
\end{picture}}%

\newcommand{\BIMO}[1]%
{\begin{picture}(#1,0)(-#1,0)%
\put(-#1,0){\vector(1,0){#1}}%
\put(-6,-2){\line(0,1){4}}%
\put(-#1,-2){\hspace{2pt}\line(0,1){4}}%
\end{picture}}%

\newcommand{\BIAR}[1]%
{\begin{picture}(#1,4)%
\put(0,0){\vector(1,0){#1}}%
\put(0,4){\vector(1,0){#1}}%
\end{picture}}%

\newcommand{\EQL}[1]%
{\begin{picture}(#1,0)%
\put(0,1){\line(1,0){#1}}%
\put(0,-1){\line(1,0){#1}}%
\end{picture}}%

\newcommand{\ADJAR}[1]%
{\begin{picture}(#1,4)%
\put(0,0){\vector(1,0){#1}}%
\put(#1,4){\vector(-1,0){#1}}%
\end{picture}}%

\newcommand{\BKAR}[1]%
{\begin{picture}(#1,0)%
\put(#1,0){\vector(-1,0){#1}}%
\end{picture}}%

\newcommand{\BKDOTAR}[1]%
{\NUMBEROFDOTS=#1%
\divide\NUMBEROFDOTS by 3%
\begin{picture}(#1,0)%
\multiput(#1,0)(-3,0){\NUMBEROFDOTS}{\circle*{1}}%
\put(0,0){\vector(-1,0){0}}%
\end{picture}}%

\newcommand{\BKMONO}[1]%
{\begin{picture}(#1,0)(-#1,0)%
\put(0,0){\vector(-1,0){#1}}%
\put(-2,-2){\line(0,1){4}}%
\end{picture}}%

\newcommand{\BKEPI}[1]%
{\begin{picture}(#1,0)%
\put(#1,0){\vector(-1,0){#1}}%
\put(6,-2){\line(0,1){4}}%
\end{picture}}%

\newcommand{\BKBIMO}[1]%
{\begin{picture}(#1,0)%
\put(#1,0){\vector(-1,0){#1}}%
\put(6,-2){\line(0,1){4}}%
\put(#1,-2){\hspace{-2pt}\line(0,1){4}}%
\end{picture}}%

\newcommand{\BKBIAR}[1]%
{\begin{picture}(#1,4)%
\put(#1,0){\vector(-1,0){#1}}%
\put(#1,4){\vector(-1,0){#1}}%
\end{picture}}%

\newcommand{\BKADJAR}[1]%
{\begin{picture}(#1,4)%
\put(0,4){\vector(1,0){#1}}%
\put(#1,0){\vector(-1,0){#1}}%
\end{picture}}%

\newcommand{\lowername}[2]%
{$\stackrel{\makebox[1pt]{#1}}%
{\begin{picture}(0,0)%
\truex{600}%
\put(0,0){\makebox(0,\value{x})[t]{\makebox[1pt]{$#2$}}}%
\end{picture}}$}%

\newcommand{\hcase}[2]%
{\makebox[0pt]%
{\raisebox{-1pt}[0pt][0pt]{#1{#2}}}}%

\newcommand{\Hcase}[3]%
{\makebox[0pt]
{\raisebox{-1pt}[0pt][0pt]%
{$\stackrel{\makebox[0pt]{$\textstyle{#2}$}}{#1{#3}}$}}}%

\newcommand{\hcasE}[3]%
{\makebox[0pt]%
{\raisebox{-9pt}[0pt][0pt]%
{\lowername{#1{#3}}{#2}}}}%

\newcommand{\hbicase}[2]%
{\makebox[0pt]%
{\raisebox{-2.5pt}[0pt][0pt]{#1{#2}}}}%

\newcommand{\Hbicase}[4]%
{\makebox[0pt]
{\raisebox{-10.5pt}[0pt][0pt]%
{$\stackrel{\makebox[0pt]{$\textstyle{#2}$}}%
{\mbox{\lowername{#1{#4}}{#3}}}$}}}%

\newcommand{\EAR}[1]%
{\begin{picture}(#1,0)%
\put(0,0){\vector(1,0){#1}}%
\end{picture}}%

\newcommand{\EDOTAR}[1]%
{\truex{100}\truey{300}%
\NUMBEROFDOTS=#1%
\divide\NUMBEROFDOTS by \value{y}%
\begin{picture}(#1,0)%
\multiput(0,0)(\value{y},0){\NUMBEROFDOTS}%
{\circle*{\value{x}}}%
\put(#1,0){\vector(1,0){0}}%
\end{picture}}%

\newcommand{\EMONO}[1]%
{\begin{picture}(#1,0)%
\put(0,0){\vector(1,0){#1}}%
\truex{300}\truey{600}%
\put(\value{x},-\value{x}){\line(0,1){\value{y}}}%
\end{picture}}%

\newcommand{\EEPI}[1]%
{\begin{picture}(#1,0)(-#1,0)%
\put(-#1,0){\vector(1,0){#1}}%
\truex{300}\truey{600}\truez{800}%
\put(-\value{z},-\value{x}){\line(0,1){\value{y}}}%
\end{picture}}%

\newcommand{\EBIMO}[1]%
{\begin{picture}(#1,0)(-#1,0)%
\put(-#1,0){\vector(1,0){#1}}%
\truex{300}\truey{600}\truez{800}%
\put(-\value{z},-\value{x}){\line(0,1){\value{y}}}%
\put(-#1,-\value{x}){\hspace{3pt}\line(0,1){\value{y}}}%
\end{picture}}%

\newcommand{\EBIAR}[1]%
{\truex{400}%
\begin{picture}(#1,\value{x})%
\put(0,0){\vector(1,0){#1}}%
\put(0,\value{x}){\vector(1,0){#1}}%
\end{picture}}%

\newcommand{\EEQL}[1]%
{\begin{picture}(#1,0)%
\truex{200}%
\put(0,\value{x}){\line(1,0){#1}}%
\put(0,0){\line(1,0){#1}}%
\end{picture}}%

\newcommand{\EADJAR}[1]%
{\truex{400}%
\begin{picture}(#1,\value{x})%
\put(0,0){\vector(1,0){#1}}%
\put(#1,\value{x}){\vector(-1,0){#1}}%
\end{picture}}%

\newcommand{\ear}%
{\hspace{\SOURCE\unitlength}%
\hcase{\EAR}{\ARROWLENGTH}}%

\newcommand{\Earv}[2]{\Hcase{\EAR}{#1}{#200}}%

\newcommand{\Ear}[1]%
{\hspace{\SOURCE\unitlength}%
\Hcase{\EAR}{#1}{\ARROWLENGTH}}%

\newcommand{\eaR}[1]%
{\hspace{\SOURCE\unitlength}%
\hcasE{\EAR}{#1}{\ARROWLENGTH}}%

\newcommand{\edotar}%
{\hspace{\SOURCE\unitlength}%
\hcase{\EDOTAR}{\ARROWLENGTH}}%

\newcommand{\Edotar}[1]%
{\hspace{\SOURCE\unitlength}%
\Hcase{\EDOTAR}{#1}{\ARROWLENGTH}}%

\newcommand{\edotaR}[1]%
{\hspace{\SOURCE\unitlength}%
\hcasE{\EDOTAR}{#1}{\ARROWLENGTH}}%

\newcommand{\emono}%
{\hspace{\SOURCE\unitlength}%
\hcase{\EMONO}{\ARROWLENGTH}}%

\newcommand{\Emono}[1]%
{\hspace{\SOURCE\unitlength}%
\Hcase{\EMONO}{#1}{\ARROWLENGTH}}%

\newcommand{\emonO}[1]%
{\hspace{\SOURCE\unitlength}%
\hcasE{\EMONO}{#1}{\ARROWLENGTH}}%

\newcommand{\eepi}%
{\hspace{\SOURCE\unitlength}%
\hcase{\EEPI}{\ARROWLENGTH}}%

\newcommand{\Eepi}[1]%
{\hspace{\SOURCE\unitlength}%
\Hcase{\EEPI}{#1}{\ARROWLENGTH}}%

\newcommand{\eepI}[1]%
{\hspace{\SOURCE\unitlength}%
\hcasE{\EEPI}{#1}{\ARROWLENGTH}}%

\newcommand{\ebimo}%
{\hspace{\SOURCE\unitlength}%
\hcase{\EBIMO}{\ARROWLENGTH}}%

\newcommand{\Ebimo}[1]%
{\hspace{\SOURCE\unitlength}%
\Hcase{\EBIMO}{#1}{\ARROWLENGTH}}%

\newcommand{\ebimO}[1]%
{\hspace{\SOURCE\unitlength}%
\hcasE{\EBIMO}{#1}{\ARROWLENGTH}}%

\newcommand{\eiso}%
{\hspace{\SOURCE\unitlength}%
\Hcase{\EAR}{\cong}{\ARROWLENGTH}}%

\newcommand{\Eiso}[1]%
{\hspace{\SOURCE\unitlength}%
\Hcase{\EAR}{\cong#1}{\ARROWLENGTH}}%

\newcommand{\eisO}[1]%
{\hspace{\SOURCE\unitlength}%
\hcasE{\EAR}{\cong#1}{\ARROWLENGTH}}%

\newcommand{\ebiar}%
{\hspace{\SOURCE\unitlength}%
\hbicase{\EBIAR}{\ARROWLENGTH}}%

\newcommand{\Ebiar}[2]%
{\hspace{\SOURCE\unitlength}%
\Hbicase{\EBIAR}{#1}{#2}{\ARROWLENGTH}}%

\newcommand{\eeql}%
{\hspace{\SOURCE\unitlength}%
\hbicase{\EEQL}{\ARROWLENGTH}}%

\newcommand{\eadjar}%
{\hspace{\SOURCE\unitlength}%
\hbicase{\EADJAR}{\ARROWLENGTH}}%

\newcommand{\Eadjar}[2]%
{\hspace{\SOURCE\unitlength}%
\Hbicase{\EADJAR}{#1}{#2}{\ARROWLENGTH}}%

\newcommand{\WAR}[1]%
{\begin{picture}(#1,0)%
\put(#1,0){\vector(-1,0){#1}}%
\end{picture}}%

\newcommand{\WDOTAR}[1]%
{\truex{100}\truey{300}%
\NUMBEROFDOTS=#1%
\divide\NUMBEROFDOTS by \value{y}%
\begin{picture}(#1,0)%
\multiput(#1,0)(-\value{y},0){\NUMBEROFDOTS}%
{\circle*{\value{x}}}%
\put(0,0){\vector(-1,0){0}}%
\end{picture}}%

\newcommand{\WMONO}[1]%
{\begin{picture}(#1,0)(-#1,0)%
\put(0,0){\vector(-1,0){#1}}%
\truex{300}\truey{600}%
\put(-\value{x},-\value{x}){\line(0,1){\value{y}}}%
\end{picture}}%

\newcommand{\WEPI}[1]%
{\begin{picture}(#1,0)%
\put(#1,0){\vector(-1,0){#1}}%
\truex{300}\truey{600}\truez{800}%
\put(\value{z},-\value{x}){\line(0,1){\value{y}}}%
\end{picture}}%

\newcommand{\WBIMO}[1]%
{\begin{picture}(#1,0)%
\put(#1,0){\vector(-1,0){#1}}%
\truex{300}\truey{600}\truez{800}%
\put(\value{z},-\value{x}){\line(0,1){\value{y}}}%
\put(#1,-\value{x}){\hspace{-3pt}\line(0,1){\value{y}}}%
\end{picture}}%

\newcommand{\WBIAR}[1]%
{\truex{400}%
\begin{picture}(#1,\value{x})%
\put(#1,0){\vector(-1,0){#1}}%
\put(#1,\value{x}){\vector(-1,0){#1}}%
\end{picture}}%

\newcommand{\WADJAR}[1]%
{\truex{400}%
\begin{picture}(#1,\value{x})%
\put(0,\value{x}){\vector(1,0){#1}}%
\put(#1,0){\vector(-1,0){#1}}%
\end{picture}}%

\newcommand{\war}%
{\hspace{\SOURCE\unitlength}%
\hcase{\WAR}{\ARROWLENGTH}}%

\newcommand{\War}[1]%
{\hspace{\SOURCE\unitlength}%
\Hcase{\WAR}{#1}{\ARROWLENGTH}}%

\newcommand{\waR}[1]%
{\hspace{\SOURCE\unitlength}%
\hcasE{\WAR}{#1}{\ARROWLENGTH}}%

\newcommand{\wdotar}%
{\hspace{\SOURCE\unitlength}%
\hcase{\WDOTAR}{\ARROWLENGTH}}%

\newcommand{\Wdotar}[1]%
{\hspace{\SOURCE\unitlength}%
\Hcase{\WDOTAR}{#1}{\ARROWLENGTH}}%

\newcommand{\wdotaR}[1]%
{\hspace{\SOURCE\unitlength}%
\hcasE{\WDOTAR}{#1}{\ARROWLENGTH}}%

\newcommand{\wmono}%
{\hspace{\SOURCE\unitlength}%
\hcase{\WMONO}{\ARROWLENGTH}}%

\newcommand{\Wmono}[1]%
{\hspace{\SOURCE\unitlength}%
\Hcase{\WMONO}{#1}{\ARROWLENGTH}}%

\newcommand{\wmonO}[1]%
{\hspace{\SOURCE\unitlength}%
\hcasE{\WMONO}{#1}{\ARROWLENGTH}}%

\newcommand{\wepi}%
{\hspace{\SOURCE\unitlength}%
\hcase{\WEPI}{\ARROWLENGTH}}%

\newcommand{\Wepi}[1]%
{\hspace{\SOURCE\unitlength}%
\Hcase{\WEPI}{#1}{\ARROWLENGTH}}%

\newcommand{\wepI}[1]%
{\hspace{\SOURCE\unitlength}%
\hcasE{\WEPI}{#1}{\ARROWLENGTH}}%

\newcommand{\wbimo}%
{\hspace{\SOURCE\unitlength}%
\hcase{\WBIMO}{\ARROWLENGTH}}%

\newcommand{\Wbimo}[1]%
{\hspace{\SOURCE\unitlength}%
\Hcase{\WBIMO}{#1}{\ARROWLENGTH}}%

\newcommand{\wbimO}[1]%
{\hspace{\SOURCE\unitlength}%
\hcasE{\WBIMO}{#1}{\ARROWLENGTH}}%

\newcommand{\wiso}%
{\hspace{\SOURCE\unitlength}%
\Hcase{\WAR}{\cong}{\ARROWLENGTH}}%

\newcommand{\Wiso}[1]%
{\hspace{\SOURCE\unitlength}%
\Hcase{\WAR}{#1}{\ARROWLENGTH}}%

\newcommand{\wisO}[1]%
{\hspace{\SOURCE\unitlength}%
\hcasE{\WAR}{#1}{\ARROWLENGTH}}%

\newcommand{\wbiar}%
{\hspace{\SOURCE\unitlength}%
\hbicase{\WBIAR}{\ARROWLENGTH}}%

\newcommand{\Wbiar}[2]%
{\hspace{\SOURCE\unitlength}%
\Hbicase{\WBIAR}{#1}{#2}{\ARROWLENGTH}}%

\newcommand{\weql}%
{\hspace{\SOURCE\unitlength}%
\hbicase{\EEQL}{\ARROWLENGTH}}%

\newcommand{\wadjar}%
{\hspace{\SOURCE\unitlength}%
\hbicase{\WADJAR}{\ARROWLENGTH}}%

\newcommand{\Wadjar}[2]%
{\hspace{\SOURCE\unitlength}%
\Hbicase{\WADJAR}{#1}{#2}{\ARROWLENGTH}}%

\newcommand{\Vcase}[3]{\makebox[0pt]%
{\makebox[0pt][r]{\raisebox{0pt}[0pt][0pt]{${#2}\hspace{2pt}$}}}#1{#3}}%

\newcommand{\vcasE}[3]{\makebox[0pt]%
{#1{#3}\makebox[0pt][l]{\raisebox{0pt}[0pt][0pt]{\hspace{2pt}$#2$}}}}%

\newcommand{\Vbicase}[4]{\makebox[0pt]%
{\makebox[0pt][r]{\raisebox{0pt}[0pt][0pt]{$#2$\hspace{4pt}}}#1{#4}%
\makebox[0pt][l]{\raisebox{0pt}[0pt][0pt]{\hspace{5pt}$#3$}}}}%

\newcommand{\SAR}[1]%
{\begin{picture}(0,0)%
\put(0,0){\makebox(0,0)%
{\begin{picture}(0,#1)%
\put(0,#1){\vector(0,-1){#1}}%
\end{picture}}}\end{picture}}%

\newcommand{\SDOTAR}[1]%
{\truex{100}\truey{300}%
\NUMBEROFDOTS=#1%
\divide\NUMBEROFDOTS by \value{y}%
\begin{picture}(0,0)%
\put(0,0){\makebox(0,0)%
{\begin{picture}(0,#1)%
\multiput(0,#1)(0,-\value{y}){\NUMBEROFDOTS}%
{\circle*{\value{x}}}%
\put(0,0){\vector(0,-1){0}}%
\end{picture}}}\end{picture}}%

\newcommand{\SMONO}[1]%
{\begin{picture}(0,0)%
\put(0,0){\makebox(0,0)%
{\begin{picture}(0,#1)%
\put(0,#1){\vector(0,-1){#1}}%
\truex{300}\truey{600}%
\put(0,#1){\begin{picture}(0,0)%
\put(-\value{x},-\value{x}){\line(1,0){\value{y}}}\end{picture}}%
\end{picture}}}\end{picture}}%

\newcommand{\SEPI}[1]%
{\begin{picture}(0,0)%
\put(0,0){\makebox(0,0)%
{\begin{picture}(0,#1)%
\put(0,#1){\vector(0,-1){#1}}%
\truex{300}\truey{600}\truez{800}%
\put(-\value{x},\value{z}){\line(1,0){\value{y}}}%
\end{picture}}}\end{picture}}%

\newcommand{\SBIMO}[1]%
{\begin{picture}(0,0)%
\put(0,0){\makebox(0,0)%
{\begin{picture}(0,#1)%
\put(0,#1){\vector(0,-1){#1}}%
\truex{300}\truey{600}\truez{800}%
\put(0,#1){\begin{picture}(0,0)%
\put(-\value{x},-\value{x}){\line(1,0){\value{y}}}\end{picture}}%
\put(-\value{x},\value{z}){\line(1,0){\value{y}}}%
\end{picture}}}\end{picture}}%

\newcommand{\SBIAR}[1]%
{\begin{picture}(0,0)%
\truex{200}%
\put(0,0){\makebox(0,0)%
{\begin{picture}(0,#1)\put(-\value{x},#1){\vector(0,-1){#1}}%
\put(\value{x},#1){\vector(0,-1){#1}}%
\end{picture}}}\end{picture}}%

\newcommand{\SEQL}[1]%
{\begin{picture}(0,0)%
\truex{100}%
\put(0,0){\makebox(0,0)%
{\begin{picture}(0,#1)\put(-\value{x},#1){\line(0,-1){#1}}%
\put(\value{x},#1){\line(0,-1){#1}}%
\end{picture}}}\end{picture}}%

\newcommand{\Sarv}[2]{\Vcase{\SAR}{#1}{#200}}%

\newcommand{\Sar}[1]{\Sarv{#1}{50}}%

\newcommand{\saRv}[2]{\vcasE{\SAR}{#1}{#200}}%

\newcommand{\saR}[1]{\saRv{#1}{50}}%

\newcommand{\Sisov}[2]%
{\Vbicase{\SAR}{#1\hspace{-2pt}}{\hspace{-2pt}\cong}{#200}}%

\newcommand{\NAR}[1]%
{\begin{picture}(0,0)%
\put(0,0){\makebox(0,0)%
{\begin{picture}(0,#1)\put(0,0){\vector(0,1){#1}}%
\end{picture}}}\end{picture}}%

\newcommand{\NDOTAR}[1]%
{\truex{100}\truey{300}%
\NUMBEROFDOTS=#1%
\divide\NUMBEROFDOTS by \value{y}%
\begin{picture}(0,0)%
\put(0,0){\makebox(0,0)%
{\begin{picture}(0,#1)%
\multiput(0,0)(0,\value{y}){\NUMBEROFDOTS}%
{\circle*{\value{x}}}%
\put(0,#1){\vector(0,1){0}}%
\end{picture}}}\end{picture}}%

\newcommand{\NMONO}[1]%
{\begin{picture}(0,0)%
\put(0,0){\makebox(0,0)%
{\begin{picture}(0,#1)%
\put(0,0){\vector(0,1){#1}}%
\truex{300}\truey{600}%
\put(-\value{x},\value{x}){\line(1,0){\value{y}}}%
\end{picture}}}%
\end{picture}}%

\newcommand{\NEPI}[1]%
{\begin{picture}(0,0)%
\put(0,0){\makebox(0,0)%
{\begin{picture}(0,#1)%
\put(0,0){\vector(0,1){#1}}%
\truex{300}\truey{600}\truez{800}%
\put(0,#1){\begin{picture}(0,0)%
\put(-\value{x},-\value{z}){\line(1,0){\value{y}}}\end{picture}}%
\end{picture}}}\end{picture}}%

\newcommand{\NBIMO}[1]%
{\begin{picture}(0,0)%
\put(0,0){\makebox(0,0)%
{\begin{picture}(0,#1)%
\put(0,0){\vector(0,1){#1}}%
\truex{300}\truey{600}\truez{800}%
\put(-\value{x},\value{x}){\line(1,0){\value{y}}}%
\put(0,#1){\begin{picture}(0,0)%
\put(-\value{x},-\value{z}){\line(1,0){\value{y}}}\end{picture}}%
\end{picture}}}\end{picture}}%

\newcommand{\NBIAR}[1]%
{\begin{picture}(0,0)%
\truex{200}%
\put(0,0){\makebox(0,0)%
{\begin{picture}(0,#1)\put(-\value{x},0){\vector(0,1){#1}}%
\put(\value{x},0){\vector(0,1){#1}}%
\end{picture}}}\end{picture}}%

\newcommand{\Nisov}[2]%
{\Vbicase{\NAR}{#1\hspace{-2pt}}{\hspace{-2pt}\cong}{#200}}%

\newcommand{\fdcase}[3]{\begin{picture}(0,0)%
\put(0,-150){#1}%
\truex{200}\truey{600}\truez{600}%
\put(-\value{x},-\value{x}){\makebox(0,\value{z})[r]{${#2}$}}%
\put(\value{x},-\value{y}){\makebox(0,\value{z})[l]{${#3}$}}%
\end{picture}}%

\newcommand{\NEDOTAR}%
{\truex{100}\truey{212}%
\NUMBEROFDOTS=5800%
\divide\NUMBEROFDOTS by \value{y}%
\begin{picture}(0,0)%
\multiput(-2900,-2900)(\value{y},\value{y}){\NUMBEROFDOTS}%
{\circle*{\value{x}}}%
\put(2900,2900){\vector(1,1){0}}%
\end{picture}}%

\newcommand{\SWAR}{\begin{picture}(0,0)%
\put(2900,2900){\vector(-1,-1){5800}}%
\end{picture}}%

\newcommand{\SWDOTAR}%
{\truex{100}\truey{212}%
\NUMBEROFDOTS=5800%
\divide\NUMBEROFDOTS by \value{y}%
\begin{picture}(0,0)%
\multiput(2900,2900)(-\value{y},-\value{y}){\NUMBEROFDOTS}%
{\circle*{\value{x}}}%
\put(-2900,-2900){\vector(-1,-1){0}}%
\end{picture}}%

\newcommand{\swaR}[1]{\fdcase{\SWAR}{}{#1}}%

\newcommand{\sdcase}[3]{\begin{picture}(0,0)%
\put(0,-150){#1}%
\truex{100}\truez{600}%
\put(\value{x},\value{x}){\makebox(0,\value{z})[l]{${#2}$}}%
\truex{300}\truey{800}%
\put(-\value{x},-\value{y}){\makebox(0,\value{z})[r]{${#3}$}}%
\end{picture}}%

\newcommand{\SEDOTAR}%
{\truex{100}\truey{212}%
\NUMBEROFDOTS=5800%
\divide\NUMBEROFDOTS by \value{y}%
\begin{picture}(0,0)%
\multiput(-2900,2900)(\value{y},-\value{y}){\NUMBEROFDOTS}%
{\circle*{\value{x}}}%
\put(2900,-2900){\vector(1,-1){0}}%
\end{picture}}%

\newcommand{\NWAR}{\begin{picture}(0,0)%
\put(2900,-2900){\vector(-1,1){5800}}%
\end{picture}}%

\newcommand{\NWDOTAR}%
{\truex{100}\truey{212}%
\NUMBEROFDOTS=5800%
\divide\NUMBEROFDOTS by \value{y}%
\begin{picture}(0,0)%
\multiput(2900,-2900)(-\value{y},\value{y}){\NUMBEROFDOTS}%
{\circle*{\value{x}}}%
\put(-2900,2900){\vector(-1,1){0}}%
\end{picture}}%

\newcommand{\nwaR}[1]{\sdcase{\NWAR}{}{#1}}%

\newcommand{\ENEAR}[2]%
{\makebox[0pt]{\begin{picture}(0,0)%
\put(0,-150){\makebox(0,0){\begin{picture}(0,0)%
\put(-6600,-3300){\vector(2,1){13200}}%
\truex{200}\truey{800}\truez{600}%
\put(-\value{x},\value{x}){\makebox(0,\value{z})[r]{${#1}$}}%
\put(\value{x},-\value{y}){\makebox(0,\value{z})[l]{${#2}$}}%
\end{picture}}}\end{picture}}}%

\newcommand{\ESEAR}[2]%
{\makebox[0pt]{\begin{picture}(0,0)%
\put(0,-150){\makebox(0,0){\begin{picture}(0,0)%
\put(-6600,3300){\vector(2,-1){13200}}%
\truex{200}\truey{800}\truez{600}%
\put(\value{x},\value{x}){\makebox(0,\value{z})[l]{${#1}$}}%
\put(-\value{x},-\value{y}){\makebox(0,\value{z})[r]{${#2}$}}%
\end{picture}}}\end{picture}}}%

\newcommand{\WNWAR}[2]%
{\makebox[0pt]{\begin{picture}(0,0)%
\put(0,-150){\makebox(0,0){\begin{picture}(0,0)%
\put(6600,-3300){\vector(-2,1){13200}}%
\truex{200}\truey{800}\truez{600}%
\put(\value{x},\value{x}){\makebox(0,\value{z})[l]{${#1}$}}%
\put(-\value{x},-\value{y}){\makebox(0,\value{z})[r]{${#2}$}}%
\end{picture}}}\end{picture}}}%

\newcommand{\WSWAR}[2]%
{\makebox[0pt]{\begin{picture}(0,0)%
\put(0,-150){\makebox(0,0){\begin{picture}(0,0)%
\put(6600,3300){\vector(-2,-1){13200}}%
\truex{200}\truey{800}\truez{600}%
\put(-\value{x},\value{x}){\makebox(0,\value{z})[r]{${#1}$}}%
\put(\value{x},-\value{y}){\makebox(0,\value{z})[l]{${#2}$}}%
\end{picture}}}\end{picture}}}%

\newcommand{\NNEAR}[2]%
{\raisebox{-1pt}[0pt][0pt]{\begin{picture}(0,0)%
\put(0,0){\makebox(0,0){\begin{picture}(0,0)%
\put(-3300,-6600){\vector(1,2){6600}}%
\truex{100}\truez{600}%
\put(-\value{x},\value{x}){\makebox(0,\value{z})[r]{${#1}$}}%
\put(\value{x},-\value{z}){\makebox(0,\value{z})[l]{${#2}$}}%
\end{picture}}}\end{picture}}}%

\newcommand{\SSWAR}[2]%
{\raisebox{-1pt}[0pt][0pt]{\begin{picture}(0,0)%
\put(0,0){\makebox(0,0){\begin{picture}(0,0)%
\put(3300,6600){\vector(-1,-2){6600}}%
\truex{100}\truez{600}%
\put(-\value{x},\value{x}){\makebox(0,\value{z})[r]{${#1}$}}%
\put(\value{x},-\value{z}){\makebox(0,\value{z})[l]{${#2}$}}%
\end{picture}}}\end{picture}}}%

\newcommand{\SSEAR}[2]%
{\raisebox{-1pt}[0pt][0pt]{\begin{picture}(0,0)%
\put(0,0){\makebox(0,0){\begin{picture}(0,0)%
\put(-3300,6600){\vector(1,-2){6600}}%
\truex{200}\truez{600}%
\put(\value{x},\value{x}){\makebox(0,\value{z})[l]{${#1}$}}%
\put(-\value{x},-\value{z}){\makebox(0,\value{z})[r]{${#2}$}}%
\end{picture}}}\end{picture}}}%

\newcommand{\NNWAR}[2]%
{\raisebox{-1pt}[0pt][0pt]{\begin{picture}(0,0)%
\put(0,0){\makebox(0,0){\begin{picture}(0,0)%
\put(3300,-6600){\vector(-1,2){6600}}%
\truex{200}\truez{600}%
\put(\value{x},\value{x}){\makebox(0,\value{z})[l]{${#1}$}}%
\put(-\value{x},-\value{z}){\makebox(0,\value{z})[r]{${#2}$}}%
\end{picture}}}\end{picture}}}%

\newcommand{\Necurve}[2]%
{\begin{picture}(0,0)%
\truex{1300}\truey{2000}\truez{200}%
\put(0,\value{x}){\oval(#200,\value{y})[t]}%
\put(0,\value{x}){\makebox(0,0){\begin{picture}(#200,0)%
\put(#200,0){\vector(0,-1){\value{z}}}%
\put(0,0){\line(0,-1){\value{z}}}\end{picture}}}%
\truex{2500}%
\put(0,\value{x}){\makebox(0,0)[b]{${#1}$}}%
\end{picture}}%

\newcommand{\Nwcurve}[2]%
{\begin{picture}(0,0)%
\truex{1300}\truey{2000}\truez{200}%
\put(0,\value{x}){\oval(#200,\value{y})[t]}%
\put(0,\value{x}){\makebox(0,0){\begin{picture}(#200,0)%
\put(#200,0){\line(0,-1){\value{z}}}%
\put(0,0){\vector(0,-1){\value{z}}}\end{picture}}}%
\truex{2500}%
\put(0,\value{x}){\makebox(0,0)[b]{${#1}$}}%
\end{picture}}%

\newcommand{\Securve}[2]%
{\begin{picture}(0,0)%
\truex{1300}\truey{2000}\truez{200}%
\put(0,-\value{x}){\oval(#200,\value{y})[b]}%
\put(0,-\value{x}){\makebox(0,0){\begin{picture}(#200,0)%
\put(#200,0){\vector(0,1){\value{z}}}%
\put(0,0){\line(0,1){\value{z}}}\end{picture}}}%
\truex{2500}%
\put(0,-\value{x}){\makebox(0,0)[t]{${#1}$}}%
\end{picture}}%

\newcommand{\Swcurve}[2]%
{\begin{picture}(0,0)%
\truex{1300}\truey{2000}\truez{200}%
\put(0,-\value{x}){\oval(#200,\value{y})[b]}%
\put(0,-\value{x}){\makebox(0,0){\begin{picture}(#200,0)%
\put(#200,0){\line(0,1){\value{z}}}%
\put(0,0){\vector(0,1){\value{z}}}\end{picture}}}%
\truex{2500}%
\put(0,-\value{x}){\makebox(0,0)[t]{${#1}$}}%
\end{picture}}%

\newcommand{\Escurve}[2]%
{\begin{picture}(0,0)%
\truex{1400}\truey{2000}\truez{200}%
\put(\value{x},0){\oval(\value{y},#200)[r]}%
\put(\value{x},0){\makebox(0,0){\begin{picture}(0,#200)%
\put(0,0){\vector(-1,0){\value{z}}}%
\put(0,#200){\line(-1,0){\value{z}}}\end{picture}}}%
\truex{2500}%
\put(\value{x},0){\makebox(0,0)[l]{${#1}$}}%
\end{picture}}%

\newcommand{\Encurve}[2]%
{\begin{picture}(0,0)%
\truex{1400}\truey{2000}\truez{200}%
\put(\value{x},0){\oval(\value{y},#200)[r]}%
\put(\value{x},0){\makebox(0,0){\begin{picture}(0,#200)%
\put(0,0){\line(-1,0){\value{z}}}%
\put(0,#200){\vector(-1,0){\value{z}}}\end{picture}}}%
\truex{2500}%
\put(\value{x},0){\makebox(0,0)[l]{${#1}$}}%
\end{picture}}%

\newcommand{\Wscurve}[2]%
{\begin{picture}(0,0)%
\truex{1300}\truey{2000}\truez{200}%
\put(-\value{x},0){\oval(\value{y},#200)[l]}%
\put(-\value{x},0){\makebox(0,0){\begin{picture}(0,#200)%
\put(0,0){\vector(1,0){\value{z}}}%
\put(0,#200){\line(1,0){\value{z}}}\end{picture}}}%
\truex{2400}%
\put(-\value{x},0){\makebox(0,0)[r]{${#1}$}}%
\end{picture}}%

\newcommand{\Wncurve}[2]%
{\begin{picture}(0,0)%
\truex{1300}\truey{2000}\truez{200}%
\put(-\value{x},0){\oval(\value{y},#200)[l]}%
\put(-\value{x},0){\makebox(0,0){\begin{picture}(0,#200)%
\put(0,0){\line(1,0){\value{z}}}%
\put(0,#200){\vector(1,0){\value{z}}}\end{picture}}}%
\truex{2400}%
\put(-\value{x},0){\makebox(0,0)[r]{${#1}$}}%
\end{picture}}%

\newcount\SCALE%

\newcount\NUMBER%

\newcount\LINE%

\newcount\COLUMN%

\newcount\WIDTH%

\newcount\SOURCE%

\newcount\ARROW%

\newcount\TARGET%

\newcount\ARROWLENGTH%

\newcount\NUMBEROFDOTS%

\newcounter{x}%

\newcounter{y}%

\newcounter{z}%

\newcounter{horizontal}%

\newcounter{vertical}%

\newskip\itemlength%

\newskip\firstitem%

\newskip\seconditem%

\newcommand{\printarrow}{}%

\newcommand{\truex}[1]{%
\NUMBER=#1%
\multiply\NUMBER by 100%
\divide\NUMBER by \SCALE%
\setcounter{x}{\NUMBER}}%

\newcommand{\truey}[1]{%
\NUMBER=#1%
\multiply\NUMBER by 100%
\divide\NUMBER by \SCALE%
\setcounter{y}{\NUMBER}}%

\newcommand{\truez}[1]{%
\NUMBER=#1%
\multiply\NUMBER by 100%
\divide\NUMBER by \SCALE%
\setcounter{z}{\NUMBER}}%

\newcommand{\changecounters}[1]{%
\SOURCE=\ARROW%
\ARROW=\TARGET%
\settowidth{\itemlength}{#1}%
\ifdim \itemlength > 2800\unitlength%
\addtolength{\itemlength}{-2800\unitlength}%
\TARGET=\itemlength%
\divide\TARGET by 1310%
\multiply\TARGET by 100%
\divide\TARGET by \SCALE%
\else%
\TARGET=0%
\fi%
\ARROWLENGTH=5000%
\advance\ARROWLENGTH by -\SOURCE%
\advance\ARROWLENGTH by -\TARGET%
\advance\SOURCE by -\TARGET}%

\newcommand{\initialize}[1]{%
\LINE=0%
\COLUMN=0%
\WIDTH=0%
\ARROW=0%
\TARGET=0%
\changecounters{#1}%
\renewcommand{\printarrow}{#1}%
\begin{center}%
\vspace{10pt}%
\begin{picture}(0,0)}%

\newcommand{\DIAG}[1]{%
\SCALE=100%
\setlength{\unitlength}{655sp}%
\initialize{\mbox{$#1$}}}%

\newcommand{\DIAGV}[2]{%
\SCALE=#1%
\setlength{\unitlength}{655sp}%
\multiply\unitlength by \SCALE%
\divide\unitlength by 100%
\initialize{\mbox{$#2$}}}%

\newcommand{\n}[1]{%
\changecounters{\mbox{$#1$}}%
\put(\COLUMN,\LINE){\makebox(0,0){\printarrow}}%
\thinlines%
\renewcommand{\printarrow}{\mbox{$#1$}}%
\advance\COLUMN by 4000}%

\newcommand{\nn}[1]{%
\put(\COLUMN,\LINE){\makebox(0,0){\printarrow}}%
\thinlines%
\ifnum \WIDTH < \COLUMN%
\WIDTH=\COLUMN%
\else%
\fi%
\advance\LINE by -4000%
\COLUMN=0%
\ARROW=0%
\TARGET=0%
\changecounters{\mbox{$#1$}}%
\renewcommand{\printarrow}{\mbox{$#1$}}}%

\newcommand{\conclude}{%
\put(\COLUMN,\LINE){\makebox(0,0){\printarrow}}%
\thinlines%
\ifnum \WIDTH < \COLUMN%
\WIDTH=\COLUMN%
\else%
\fi%
\setcounter{horizontal}{\WIDTH}%
\setcounter{vertical}{-\LINE}%
\end{picture}}%

\newcommand{\diag}{%
\conclude%
\raisebox{0pt}[0pt][\value{vertical}\unitlength]{}%
\hspace*{\value{horizontal}\unitlength}%
\vspace{10pt}%
\end{center}%
\setlength{\unitlength}{1pt}}%

\newcommand{\diagv}[3]{%
\conclude%
\NUMBER=#1%
\rule{0pt}{\NUMBER pt}%
\hspace*{-#2pt}%
\raisebox{0pt}[0pt][\value{vertical}\unitlength]{}%
\hspace*{\value{horizontal}\unitlength}
\NUMBER=#3%
\advance\NUMBER by 10%
\vspace*{\NUMBER pt}%
\end{center}%
\setlength{\unitlength}{1pt}}%

\newcommand{\N}[1]%
{\raisebox{0pt}[7pt][0pt]{$#1$}}%

\newcommand{\crosslength}[2]{%
\settowidth{\firstitem}{#1}%
\settowidth{\seconditem}{#2}%
\ifdim\firstitem < \seconditem%
\itemlength=\seconditem%
\else%
\itemlength=\firstitem%
\fi%
\divide\itemlength by 2%
\hspace{\itemlength}}%

%% file: fc.tex
\section{The Formal Calculus of Matrix Pairs}

\noindent This section is devoted to a development of the formal matrix calculus defined in the Introduction. Let us inspect more closely the definition of the relation $\leq_n$ on $L'_n (R).$ Recall that the elements $(B \; | \; A)$ of 
$L'_n (R)$ are given by matrices $A$ and $B$ (with entries in $R$) that have the same number of columns. The number of columns in $A$ is $n,$ but neither dimension of $B$ is specified. To better keep track of the computations, we will assign dimensions so that $B$ is an $m \times k$ matrix. In that case, the dimensions of $A$ are $m \times n.$ We are given that
$$(B \; |\; A) \leq_n (B' \; |\; A')$$
if and only if there exist matrices $U,$ $V$ and $G,$ of appropriate size, such that 
\begin{equation} \label{Prest relations}
UB = B'V \;\;\; \mbox{and} \;\;\; UA = A' + B'G.
\end{equation}
Let $B'$ be an $m' \times k'$ matrix. Then we see that $U$ is an $m' \times m$ matrix, $V$ a $k' \times k$ matrix, and
$G$ a $k' \times n$ matrix. 

\subsection{Preliminary Observations.}

\noindent Let us begin by verifying some of the claims made in the Introduction. When it is clear from the context, we will drop the subscript in the notation $\leq_n.$

\begin{proposition} \label{transitive}
The relation $\leq$ on $L'_n (R)$ is a pre-order (reflexive and transitive). 
\end{proposition}

\begin{proof}
The reflexive property $(B \; | \; A) \leq (B \; | \; A)$ is established by letting $U = I_m,$ $V = I_n,$ and 
$G = {_k}\zero_n.$ To verify transitivity, suppose that $(B_1 \; | \; A_1) \leq (B_2 \; | \; A_2)$ and \linebreak
$(B_2 \; | \; A_2) \leq (B_3 \; | \; A_3),$ with corresponding matrices $U_i,$ $V_i$ and $G_i,$ $i = 1,2.$ Thus we have
$U_1B_1 = B_2V_1$ and $U_2B_2 = B_3V_2,$ which implies
$$(U_2U_1)B_1 = U_2B_2V_1 = B_2(V_2V_1).$$
Let $U = U_2 U_1$ and $V = V_2 V_1.$ Then
$$UA_1 = U_2 U_1 A_1 = U_2 (A_2 + B_2G_1) = U_2A_2 + U_2B_2G_1 = A_3 + B_3 G_2 + B_3V_2G_1 = A_3 + B_3 G,$$
where $G = G_2 + V_2G_1.$ 
\end{proof}

Let us isolate the r\^{o}les played by each of the three parameters $U,$ $V$ and $G$ in the definition of the pre-order on $L'_n (R)$ (cf.\ the comments preceding Lemma 1.1.13 of~\cite{Mike's book}).

\begin{theorem} \label{rules} 
Let $(B \; | \; A) \in L'_n (R),$ where $B$ is an $m \times k$ matrix.
\begin{enumerate}
\item If $U$ is a matrix with $m$ columns, then 
$(B \; | \; A) \leq_n (UB \; | \; UA).$ 
\item If $V$ is a matrix with $k$ rows, then 
$(BV \; | \; A) \leq_n (B \; | \; A).$
\item If $G$ is a $k \times n$ matrix, then 
$(B \; | \; A + BG) \leq_n (A \; | \; B).$
\end{enumerate}
The relation $\leq_n$ is the least pre-order on $L'_n (R)$ satisfying (1), (2) and (3).
\end{theorem}

\begin{proof}
For (1), let $V = I_k$ and $G = {_k}\zero_n;$ for (2), let $U = I_m$ and $G = {_k}\zero_n;$ and for (3), let $U = I_m$ and $V = I_k.$ Suppose that $\prec$ is a partial order on $L'_n (R)$ satisfying (1), (2) and (3). If 
$(B \; | \; A) \leq (B' \; | \; A')$ and $U,$ $V,$ and $G$ satisfy Equations~\ref{Prest relations}, then
$$(B \; | \; A) \prec (UB \; | \; UA) = (B'V \; | \; A'+B'G) \prec (B' \; | \; A'+B'G) \prec (B' \; | \; A').$$
\end{proof}

In the sequel, the three statements of Theorem~\ref{rules} are referred to as the Rules of Divisibility (RoD) (1)-(3).
The partial order $\leq$ imposes on $L'_n (R)$ the equivalence relation 
$(B \; | \; A) \approx_n (B' \; | \; A')$ defined to hold provided $(B \; | \; A) \leq_n (B' \; | \; A')$ and 
$(B \; | \; A) \leq_n (B' \; | \; A').$ The equivalence class of $(B \; | \; A)$ will be denoted by 
$[B \; | \; A].$\bigskip

\noindent {\bf Definition.} An {\em $n$-ary matrix pair} is an equivalence class $[B \; | \; A]$ of some element 
$(B \; | \; A)$ of $L'_n (R)$ modulo the relation $\approx_n$. The set of $n$-ary matrix pairs is denoted by $L_n (R).$ The partial order $\leq_n$ on $L'_n (R)$ induces a partial order (reflexive, symmetric and transitive) on $L_n (R),$ which will also be denoted $\leq_n,$ (or $\leq,$ when there is no danger of confusion.) \bigskip

The next corollary indicates those situations in which the Rules of Divisibility yield equality.

\begin{corollary} \label{invertible}
Let $[B \; | \; A] \in L_n (R),$ where $B$ is an $m \times k$ matrix. 
\begin{enumerate}
\item If $P$ is an invertible $m \times m$ matrix, then 
$[B \; | \; A] = [PB \; | \; PA].$
\item If $Q$ is an invertible $k \times k$ matrix, then
$[B \; | \; A] = [BQ \; | \; A].$
\item If $G$ is a $k \times n$ matrix, then $[B \; | \; A + BG] = [B \; | \; A].$
\end{enumerate}
\end{corollary}

\begin{proof} (1). By RoD (1), $[B \; | \; A] \leq [PB \; | \; PA].$ Replacing $P$ by $P^{-1}$ yields $[PB \; | \; PA] \leq [B \; | \; A].$ \smallskip

\noindent (2). By RoD (2), $[BQ \; | \; A] \leq [B \; | \; A].$ Replacing $Q$ by $Q^{-1}$ yields $[B \; | \; A] \leq [BQ \; | \; A].$ \smallskip

\noindent (3). By RoD (3), $[B \; | \; A + BG] \leq [B \; | \; A].$ Replacing $G$ by $-G$ yields $[B \; | \; A] \leq [B \; | \; A + BG].$
\end{proof}

If we take the matrices $P$ or $Q$ in Corollary~\ref{invertible} to be permutation matrices, we see that permutating (simultaneously) the rows of $A$ and $B,$ or the columns of $B,$ does not change the matrix pair. Similarly, if $P$ or $Q$ are taken to be elementary matrices, then Corollary~\ref{invertible} implies that we may perform elementary row operations (simultaneously) on $A$ and $B,$ or elementary column operations on $B,$ without changing the matrix pair. The next proposition describes the affect of adding rows to an $n$-ary matrix pair. 

\begin{lemma} \label{more rows}
Let $[B \; | \; A]$ and $[B' \; | \; A']$ belong to $L_n (R),$ with $B$ an $m \times k$ matrix, and $B'$ an $m' \times k$ matrix. Then
$$\left[ \begin{array}{c|c}
B & A \\
B' & A' \end{array} \right]  \leq [B \; | \; A].$$
If $A'$ and $B'$ are zero matrices, then equality holds.
\end{lemma}

\begin{proof}
For the first part, apply RoD (1) with $U = (I_m, {_m}\zero_{m'}).$ This yields the inequality in $L'_n (R),$ and therefore in $L_n (R).$ If $A'$ and $B'$ are zero matrices, then the inequality in the opposite direction also follows using RoD (1), but with ${\dis U = \left( \begin{array}{c} I_m \\ {_{m'}}\zero_m \end{array} \right).}$
\end{proof}

If an $n$-ary matrix pair has an extra row of zeros, then that row may be removed, without changing the matrix pair. For, we may permute the rows of the matrix pair to put the zero row at the bottom, and then apply Lemma~\ref{more rows} to remove it. The next proposition describes the situation when columns are added to the matrix $B.$ 

\begin{lemma} \label{more columns}
Let $[B \; | \; A] \in L_n,$ with $B$ an $m \times k$ matrix. If $B'$ is an $m \times k'$ matrix, then  
$$[B \; | \; A] \leq [B, B' \; | \; A].$$ If $B'$ is the zero matrix, then equality holds.
\end{lemma}

\begin{proof} The first assertion follows from RoD (2) with 
${\dis V = \left( \begin{array}{c}
I_k \\ {_{k'}}\zero_k \end{array} \right).}$ 
If $B'$ is a zero matrix, then the inequality in the opposite direction also follows using RoD (2), but with 
$V = (I_k, {_k}\zero_{k'}).$
\end{proof}

Lemma~\ref{more columns} implies that an extra column of zeros may be removed from the left matrix $B$ without changing the $n$-ary matrix pair. The next proposition shows how to compute the infimum of two $n$-ary matrix pairs.

\begin{proposition} \label{infimum}
The infimum of two elements $[B \; | \; A]$ and $[B' \; | \; A']$ of $L_n (R)$ is given by 
$$[B \; | \; A] \wedge [B' \; | \; A'] := 
\left[ \begin{array}{cc|c}
B & \zero & A \\ \zero & B' & A' \end{array} \right].$$
\end{proposition}

\begin{proof}
By Lemmata~\ref{more rows} and~\ref{more columns}, it is clear that 
$[B \; | \; A] \wedge [B' \; | \; A'] \leq [B \; | \; A]$ and 
$[B \; | \; A] \wedge [B' \; | \; A'] \leq [B' \; | \; A'].$ 
So suppose that $(B'' \; | \; A'') \leq (B \; | \; A)$ and $(B'' \; | \; A'') \leq (B' \; | \; A')$ in $L'_n (R).$ 
There are $U,$ $V$ and $G$ such that $UB'' = BV$ and $UA'' = A + BG;$ and there are $U',$ $V'$ and $G'$ such that $U'B'' = B'V'$ and $U'A'' = A' + B'G'.$ To see that 
$(B'' \; | \; A'') \leq (B \; | \; A) \wedge (B' \; | \; A'),$ just note that
$$\left( \begin{array}{c} U \\ U' \end{array} \right) B'' =
\left( \begin{array}{c} BV \\ B'V' \end{array} \right) =
\left( \begin{array}{ll} B & \zero \\ \zero & B' \end{array} \right)
\left( \begin{array}{c} V \\ V' \end{array} \right)$$
and 
$$\left( \begin{array}{c} U \\ U' \end{array} \right) A'' =
\left( \begin{array}{c} UA'' \\ U'A'' \end{array} \right) =
\left( \begin{array}{c} A + BG \\ A' + B'G' \end{array} \right) =
\left( \begin{array}{c} A \\ A' \end{array} \right) +
\left( \begin{array}{ll} B & \zero \\ \zero & B' \end{array} \right)
\left( \begin{array}{c} G \\ G' \end{array} \right).$$
\end{proof}

\begin{example} If $D$ is a diagonal matrix with diagonal entries $d_{ii},$ then Proposition~\ref{infimum} implies that 
$$[D \; | \; A] = \bigwedge_i [d_{ii} \; | {_i}A],$$
where ${_i}A$ denotes the $i$-th row of $A.$ This is relevant, for example, in the case when $R$ is a commutative PID.
By Theorem 3.8 of~\cite{Jac}, the $m \times k$ matrix $B$ may be {\em diagonalized} in the sense that there are invertible matrices $P$ and $Q$ such that $PBQ = D$ is a diagonal matrix. Applying RoD (1) and (2) with the invertible matrices $P$ and $Q,$ respectively, we see that for any $n$-ary matrix pair,
$$[B \; | \; A] = [PB \; | \; PA] = [PAQ \; | \; PA] = [D \; | \; PA]$$
is the infimum of $n$-ary matrix pairs $[d_{ii} \; | \; {_i}(PA)],$ where the matrix on the left is a $1 \times 1$ scalar matrix. 
\end{example}

The partial order $L_n (R)$ has a minimum element $0_n.$ For if $[B \; | \; A] \in L_n$ with $B$ an $m \times k$ matrix, then
$$[{_n}\zero_1 \; | \; I_n] \leq [{_m}\zero_1 \; | \; A] = [B \cdot {_k}\zero_1 \; | \; A] \leq [B \; | \; A].$$ 
The first inequality follows by RoD (1) with $U = A;$ the second by RoD (2) with $V ={_k}\zero_1.$ 

\begin{proposition} \label{minimum}
An element $[B \; | \; A] \in L_n (R)$ is the minimum element $0_n$ if and only if there exists a matrix $U$ such that
$UB = {_n}\zero_k$ and $UA = I_n.$
\end{proposition}

\begin{proof}
If $[B \; | \; A]$ satisfies the condition, then $[B \; | \; A] \leq [{_n}\zero_k \; | I_n],$ by (1) of Theorem~\ref{rules} with $U$ as given. But $[{_n}\zero_k \; | I_n] = 0_n,$ by Lemma~\ref{more columns}.

For the converse, suppose that $(B \; | \; A) \leq ({_n}\zero_1 \; | \; I_n).$ Then there exist matrices $U,$ $V,$ and $G$ such that $UB = {_n}\zero_1 \cdot V = {_n}\zero_k$ and $UA = I_n + {_n}\zero_1 \cdot G = I_n.$
\end{proof}

The partial order $L_n (R)$ also has a maximum element $1_n.$ For if $[B \; | \; A] \in L_n$ with $B$ an $m \times k$ matrix, then
$$[B \; | \; A] = [I_m \cdot B \; | \; A] \leq [I_m \; | \; A].$$
This matrix pair $[I_m \; | \; A]$ will be seen to be the maximum element once it is put into a form independent of the matrix $A$ and dimension $m.$ By RoD (3) with $G = -A,$  
$$[I_m \; | A] = [I_m \; | \; {_m}\zero_n] = \bigwedge [1 \; | \; {_1}\zero_n] = [1 \; | \; {_1}\zero_n].$$
The second equality follows from Proposition~\ref{infimum} and Lemma~\ref{more columns}. The maximum element is therefore given by $$1_n = [1 \; | \; {_1}\zero_n].$$
There are other forms $[B \; | \; A]$ that represent the maximum element. For example, if $B$ is any $m \times k$ matrix, then $1_n  = [I_k \; | \; {_k}\zero_n] \leq [B \; | \; {_m}\zero_n],$ which is obtained using RoD (1) with $U = B.$ More generally, we have the following.  

\begin{proposition} \label{maximum}
An element $[B \; | \; A] \in L_n (R)$ is the maximum element $1_n$ if and only if there exists a matrix $W$ such that
$BW = A.$
\end{proposition}

\begin{proof}
Suppose that $A = BW$ for some $k \times n$ matrix $W.$ Then
$$[B \; | \; A] = [B \; | \; BW] = [B \; | \; {_m}\zero_n] = 1_n,$$
by RoD (3) with $G = - W.$ On the other hand, if $[B \; | \; A] = 1_n,$ then 
$(1 \; | \; {_1}\zero_n) \leq (B \; | \; A),$ so there exist matrices $U,$ $V,$ and $G$ such that $U \cdot 1 = BV$ and $U \cdot {_1}\zero_n = A + BG.$ Then $A = B(-G),$ and so $W = - G.$
\end{proof}

To aid our intuition, we may construe the symbol $(B \; | \; A)$ as the proposition ''$B$ divides $A$ on the left.'' 
This proposition is then assigned a position in the partial order $L_n (R),$  governed by the Rules of Divisibility in $L'_n (R).$ These rules correspond to the matrix operations that preserve the relation ''$B$ divides $A$ on the left.'' Proposition~\ref{maximum} asserts that the proposition $[B \; | \; A]$ is assigned the maximum value $1_n$ if and only if, $B|A$ is {\em true,} that is if $B$ {\em does,} in fact, divide $A$ on the left: $$[B \; | \; A] = 1_n  \mbox{  if and only if  }  B|A.$$

\subsection{Homogeneous Systems}

An element of $L'_n (R)$ is called a ({\em homogeneous}) {\em system} ({\em of linear equations in $n$ variables}) if it is of the form $(\zero \; | A);$ the {\em matrix of coefficients} of this system is $m \times n$ matrix $A.$ For simplicity, we will denote such a system by $(A) := (\zero \; | \; A),$ and its $n$-ary matrix pair by $[A] := [\zero \; | \; A].$ For example, the maximum element $1_n = [{_1}\zero_n]$ and minimum element $0_n = [I_n]$ are both matrix pairs of systems. The basic properties of systems are given in the following proposition.

\begin{proposition} \label{systems}
Let $A$ be an $m \times n$ matrix. Then the following hold:
\begin{enumerate}
\item if $A'$ is another matrix with $n$ columns, then $(A) \approx_n (A')$ in $L'_n (R)$ if and only if there exist matrices $U$ and $U'$ such that $UA = A'$ and $U'A' = A;$
\item $(A) = (B' \; | \; A')$ in $L'_n (R)$ if and only if there exist matrices $U,$ $W,$ and $G$ such that
$$UB' = \zero; \;\;\; UA' = A; \mbox{  and  } WA = A' + B'G;$$
\item suppose that $A'$ is an $m \times k$ matrix, with $(A, A') \leq_{n+k} (B \; | \; C', C'')$ in $L'_{n+k}(R),$ where $C'$ and $C''$ have $n$ (resp., $k$) columns. Then
$(A \; | \; A') \leq_k (B, C' \; | \; C'').$ 
\end{enumerate}
\end{proposition} 

\begin{proof}
Both ($1$) and ($2$) follow from the definition. To prove ($3$), suppose that we have matrices $U,$ $V,$ and 
$G = (G', G'')$ such that 
$$U \cdot \zero = BV \mbox{ and } U(A, A') = (C', C'') + B(G', G'').$$
The two equations $UA = C' + BG'$ and $UA' = C'' + BG''$ may be rewritten as
$$UA = (B, C') \left( \begin{array}{c} G' \\ I_n \end{array} \right) \mbox{  and  }  
UA' = C'' + (B,C') \left( \begin{array}{c} G'' \\ \zero \end{array} \right).$$
\end{proof}

\noindent If $A$ and $A'$ satisfy Condition (1) of Proposition~\ref{systems}, then it is immediate
that for every left $R$-module ${_R}M,$ the respective solution subgroups of $({_R}M)^n$ of the corresponding homogeneous systems of linear equations (in the classical sense) are equal.   

A matrix $B$ is called {\em regular} if there exists a matrix $C$ such that $BCB = B.$ 

\begin{theorem} \label{regular matrix}
The following are equivalent for an $m \times k$ matrix $B:$
\begin{enumerate}
\item it is regular;
\item for every $m \times n$ matrix $A,$ there exists a matrix $A'$ with $n$ columns such that 
$$[B \; | \; A] = [A'];$$
\item there exists a matrix $A'$ with $n$ columns such that $[B \; | \; I_n] = [A'].$
\end{enumerate}
\end{theorem}

\begin{proof}
(1) $\Rightarrow$ (2). Suppose that $B$ is regular. By RoD (2) with $V = CB$ and $V = C,$ respectively, 
Then $$[B \; | \; A] = [BCB \; | \; A] \leq [BC \; | \; A] \leq [B \; | \; A].$$ 
Then $(BC)^2 = (BC)(BC) = (BCB)C = BC$ and we may replace $B$ by the $n \times n$ idempotent matrix $E = BC$ in the $n$-ary matrix pair. Then
$$[E \; | \; A] \leq [ \zero \; | (I_n - E)A] \leq [E \; | \; A - EA] \leq [E \; | \; A].$$
The first inequality follows from RoD (1) with $U = I_n - E;$ the second from RoD (2) with $V = \zero;$ and the third
from RoD (3) with $G = A.$ Now let $A' = (I_n - E)A.$

(3) $\Rightarrow$ (1). We are given that $(B \; | \; I_n) \leq (\zero \; | \; A')$ and 
$(\zero \; | \; A') \leq (B \; | \; I_n).$ From the first inequality, we obtain a matrix $U$ such that $UB = \zero$ and $U = A.$ In short, $AB = \zero.$ From the second inequality, we obtain matrices $U'$ and $G'$ such that 
$U'A = I_n + BG'.$ Multiplying on the right by $B$ yields the equation $\zero = B + BG'B.$ Whence $B = BCB$ with 
$C = -G'.$
\end{proof}

The ring $R$ is {\em von Neumann regular} if for every $r \in R,$ there exists an $s \in R$ such that $rsr = r.$
By Theorem 1.7 of~\cite{Good}, every matrix over a von Neumann regular ring is regular.

\begin{corollary} \label{von Neumann}
The following are equivalent for the ring $R:$
\begin{enumerate}
\item it is von Neumann regular;
\item for every $n \geq 1,$ every element $(B \; | \; A)$ of $L'_n (R)$ is equivalent to a system;
\item every element $(B \; | \; A)$ of $L'_1 (R)$ is equivalent to a system;
\end{enumerate}
\end{corollary}

\begin{proof}
To prove (1) $\Rightarrow$ (2) just note that $B$ is a regular matrix, and apply Theorem~\ref{regular matrix}. For (3) $\Rightarrow$ (1), let $r \in R,$ and apply the hypothesis and Theorem~\ref{regular matrix} to 
$(r \; | \; 1) \in L'_1 (R).$ 
\end{proof}

\begin{example} \label{field}
Suppose that $R = k$ is a field. Evidently, it is von Neumann regular, so that every $n$-ary matrix pair is the class of some system $(A)$ of linear equations. By Corollary~\ref{invertible}, we may choose the matrix of coefficients $A$ to be a row reduced echelon matrix. By Proposition~\ref{systems}.(1), this choice of $A$ is unique. Therefore, the $n$-ary matrix pairs in $L_n (k)$ are in bijective correspondence with row reduced echelon matrices with $n$ columns.
\end{example}

Example~\ref{field} together with Proposition~\ref{systems} indicates that the formal matrix calculus presented reduces over $k$ to the study of homogeneous systems of linear equations.

\subsection{Duality.}

The opposite ring of $R$ is denoted $R^{\op}.$ Its elements are those of $R,$ as is the underlying abelian group structure, but the multiplication $\ast$ in $R^{\op}$ is given by $r \ast s = sr,$ where the multiplication on the right is carried out in $R.$ More generally, multiplication of matrices over $R^{\op}$ is denoted $A \ast B.$ It is related to multiplication of matrices over $R$ by the equation $$(A \ast B)^{\tr} = (B^{\tr}) (A^{\tr}).$$

\begin{theorem} \label{pp-duality}
If $(B \; | \; A) \leq (B' \; | \; A')$ in $L'_n (R),$ then in $L'_n (R^{\op}),$	
$$\left( \begin{array}{c|c}
(B')^{\tr} & \zero \\ (A')^{\tr} & I_n  \end{array} \right) \leq
\left( \begin{array}{c|c}
B^{\tr} & \zero \\ A^{\tr} & I_n  \end{array} \right).$$
\end{theorem}

\begin{proof} 
We are given matrices $U,$ $V$ and $G$ such that $UB = B'V$ and $UA = A' + B'G.$ In short,
$$(B', A') \left( \begin{array}{cc} V & G \\ \zero & I_n \end{array} \right) = U(B, A).$$	
In $L_n (R^{\op}),$ this yields
$$\left( \begin{array}{cc} V^{\tr} & \zero \\ G^{\tr} & I_n \end{array} \right) \ast 
\left( \begin{array}{c} (B')^{\tr} \\ (A')^{tr} \end{array} \right) = 
\left( \begin{array}{c} B^{\tr} \\ A^{tr} \end{array} \right) \ast U^{\tr}.$$
But also note that
$$\left( \begin{array}{cc} V^{\tr} & \zero \\ G^{\tr} & I_n \end{array} \right) \ast 
\left( \begin{array}{c} \zero \\ I_n \end{array} \right) = \left( \begin{array}{c} \zero \\ I_n \end{array} \right).$$
Letting ${\dis U' = \left( \begin{array}{cc} V^{\tr} & \zero \\ G^{\tr} & I_n \end{array} \right),}$ 
$V' = U^{\tr},$ and $G' = \zero$ establishes the assertion.
\end{proof}

Theorem~\ref{pp-duality} implies that the rule given by
$$[B \; | \; A] \mapsto [B \; | \; A]^* := 
\left[ \begin{array}{c|c}	B^{\tr} & \zero \\ A^{\tr} & I_n   \end{array} 	\right].$$
is a well-defined anti-morphism from the partial order $L_n (R)$ to the partial order $L_n (R^{\op}).$
There exists a similarly defined function in the opposite direction, from $L_n (R^{\op})$ to $L_n (R).$ It is also denoted by $[B \; | \; A] \mapsto [B \; | \; A]^*$ with $A$ and $B$ matrices over $R^{\op}.$ Let us verify that these
maps are mutual inverses. First note that
$$[B \; | \; A]^{**} = \left[ \begin{array}{c|c} B^{\tr} & \zero \\	A^{\tr} & I_n  \end{array} 	\right]^* =
\left[ \begin{array}{cc|c} B & A & \zero \\ \zero & I_n & I_n \end{array} \right].$$
Multiplying the bottom ''row'' by $A$ on the left, and substracting from the top ''row'' yields the equation
$$\left[ \begin{array}{cc|c} B & A & \zero \\ \zero & I_n & I_n \end{array} \right] =
\left[ \begin{array}{cc|c} B & \zero & -A \\ \zero & I_n & I_n \end{array} \right].$$
This is just the infimum $[I_n \; | \; I_n] \wedge [B \; | \; -A] = [B \; | \; -A].$ But 
$[B \; | \; -A] = [B \; | \; A]$ as a result of multiplying both $A$ and $B$ by $-I_m$ on the left and then multiplying $B$ by $-I_k$ on the right.

\begin{corollary} \label{lattice duality}
The map $[B \; | \; A] \mapsto [B \; | \; A]^*$ is an anti-isomorphism between $L_n (R)$ and $L_n (R^{\op}).$
\end{corollary}

Consider the $n$-ary matrix pair $[A]$ associated to a system. According to this anti-isomorphism, its dual in 
$L_n (R^{\op})$ is given by
$$[A]^* = [\zero \; | \; A]^* = [A^{\tr} \; | \; I_n].$$
This observation lends importance to the family of $n$-ary matrix pairs in $L_n (R)$ of the form $[B \; | \; I_n]$ for some $n \times k$ matrix $B.$ It is the family dual, in the sense of this anti-isomorphism, to the family of $n$-ary matrix pairs associated to systems in $L'_n (R).$ Condition (3) of Theorem~\ref{regular matrix} describes those $n$-ary matrix pairs that belong to the intersection of these two families. 

The supremum operation in $L_n (R)$ may then be given in terms of the infimum operation in $L_n (R^{\op}):$ 
$$[B \; | \; A] + [B' \; | \; A'] := ([B \; | \; A]^* \wedge [B' \; | \; A']^*)^*.$$
It is readily computed as
$$[B \; | \; A] + [B' \; | \; A'] := \left[ \begin{array}{cccc|c}
B & A & \zero & \zero & \zero \\
\zero & \zero & B' & A' & \zero \\
\zero & I_n & \zero & I_n & I_n  \end{array} \right].$$
	
With the infimum and supremum operations now both defined, the partial order $L_n (R)$ acquires the structure of a 
{\em lattice} with minimum and maximum elements. Recall that a lattice is {\em modular} if $a \leq b$ implies that
$(a+c) \wedge b = a + (b \wedge c).$

\begin{theorem} \label{modular}
The partial order $L_n (R)$ is a modular lattice with maximum and minimum elements.
\end{theorem} 

\begin{proof}
Let us verify the equation for modularity with $a = [B \; | \; A],$ $b = [B' \; | \; A'],$ and $c = [B'' \; | \; A''].$ We are given that $a \leq b$ so there exist matrices $U,$ $V,$ and $G$ such that $UB = B'V$ and $UA = A' + B'G.$ Let us apply some operations (explained below) to the $n$-ary matrix pair $(a+c) \wedge b:$
\begin{eqnarray*}
\left[ \begin{array}{ccccc|c}
B & A & \zero & \zero & \zero & \zero \\
\zero & \zero & B'' & A'' & \zero & \zero \\
\zero & I_n & \zero & I_n & \zero & I_n \\
\zero & \zero & \zero & \zero & B' & A'
\end{array} \right] & = &
\left[ \begin{array}{ccccc|c}
B & A & \zero & \zero & \zero & \zero \\
\zero & \zero & B'' & A'' & \zero & \zero \\
\zero & I_n & \zero & I_n & \zero & I_n \\
UB & UA & \zero & \zero & B' & A'
\end{array} \right] \\
& = & 
\left[ \begin{array}{ccccc|c}
B & A & \zero & \zero & \zero & \zero \\
\zero & \zero & B'' & A'' & \zero & \zero \\
\zero & I_n & \zero & I_n & \zero & I_n \\
B'V & A' + B'G & \zero & \zero & B' & A'
\end{array} \right] \\
& = & 
\left[ \begin{array}{ccccc|c}
B & A & \zero & \zero & \zero & \zero \\
\zero & \zero & B'' & A'' & \zero & \zero \\
\zero & I_n & \zero & I_n & \zero & I_n \\
\zero & A' & \zero & \zero & B' & A'
\end{array} \right].
\end{eqnarray*}
In the first equality, we multiplied the top ''row'' on the left by $U$ and added it to the botom. In the last equality, we subtracted right multiples, by $V$ and $G,$ respectively, of the fifth column from the first and second, respectively. Now we multiply the third ''row'' by $A'$ (on the left) and substract from the bottom ''row'' to get
\begin{eqnarray*}
\left[ \begin{array}{ccccc|c}
B & A & \zero & \zero & \zero & \zero \\
\zero & \zero & B'' & A'' & \zero & \zero \\
\zero & I_n & \zero & I_n & \zero & I_n \\
\zero & \zero & \zero & -A' & B' & \zero
\end{array} \right] & = & 
\left[ \begin{array}{ccccc|c}
B & A & \zero & \zero & \zero & \zero \\
\zero & \zero & B'' & A'' & \zero & \zero \\
\zero & I_n & \zero & I_n & \zero & I_n \\
\zero & \zero & \zero & A' & B' & \zero
\end{array} \right].
\end{eqnarray*}
The equality follows by multiplying the bottom ''row'' and fifth column by $-I.$ All that remains is to permute simultaneously some rows of both matrices, and some columns of the left matrix to obtain
$$ \left[ \begin{array}{ccccc|c}
B & A & \zero & \zero & \zero & \zero \\
\zero & \zero & B' & \zero & A' & \zero \\
\zero & \zero & \zero & B'' & A'' & \zero \\
\zero & I_n & \zero & \zero & I_n & I_n 
\end{array} \right] = a + (b \wedge c).$$
\end{proof}

\subsection{Morphisms of Rings.}

Let $f : R \to S$ be a morphism of rings. If $A$ is a matrix with entries in $R,$ denote by $f (A)$ the matrix over $S,$ of the same dimensions, obtained by applying $f$ to the entries of $A.$ Then it is easy to see that if 
$(B \; | \; A) \leq (B' \; |\; A')$ in $L'_n (R),$ then 
$(f(B) \; | \;  f(A)) \leq (f(B') \; |\; f(A'))$ in $L'_n (S).$ This is because the three matrices $U,$ $V,$ and $G$ that arise from the former inequality are taken by $f$ to matrices $f(U),$ $f(V),$ and $f(G)$ that establish the latter. This induces a morphism of partial orders 
$$[B \; | \; A] \mapsto [f(B) \; | \; f(A)],$$ 
which is denoted by $L_n (f) : L_n (R) \to L_n (S).$

Recall that a morphism $f : R \to S$ is an {\em epimorphism} if whenever ring morphisms $g,$ $h : S \to T$ are given such that $gf = hf,$ then $g = h.$ Silver~\cite{Sil} and Mazet~\cite{Maz} proved that a morphism $f: R \to S$ of rings is an epimorphism if and only if every element $s \in S,$ considered as a $1 \times 1$ matrix, may be factored as
$$XPY = s,$$ where $X,$ $P,$ and $Y$ are matrices of appropriate size, such that the entries of $XP,$ $P,$ and $PY$ all lie in the image of $f.$ We shall require the following slight strengthening of their result.

\begin{lemma}
A morphism $f : R \to S$ of rings is an epimorphism if and only if every matrix $A$ with entries in $S$ has a factorization $A = XPY,$ such that $XP,$ $P$ and $PY$ have entries in the image of $f.$
\end{lemma} 

\begin{proof}
The corresponding morphism $M_n (f) : M_n (R) \to M_n (S)$ of $n \times n$ matrix rings is also a ring epimorphism, so if $A$ is a square matrix, say $n \times n,$ then the theorem of Mazet and Silver applied to $M_n (f)$ proves the claim.
Suppose now that $A$ is an $m \times n$ matrix, where $m > n.$ Let $k = m - n,$ and apply the foregoing to the square matrix $(A, {_m}\zero_{k}).$ We obtain a Silver-Mazet factorization
$$(A, {_m}\zero_{k}) = XPY = XP (Y', Y'')$$
where $Y = (Y', Y'')$ has been decomposed so that $Y'$ has $n$ columns and $Y''$ has $k$ columns. Thus $XP,$ $P,$ and 
$PY = P(Y', Y'')$ have entries in $R.$ But then $A = XPY'$ and $PY'$ also has entries from $R.$ The case when $m < n$ is handled similarly.  
\end{proof}

\begin{theorem} \label{epi}
The following are equivalent for a morphism $f : R \to S$ of rings:
\begin{enumerate}
\item it is an epimorphism; 
\item for every matrix $A,$ the system $[A]$ is in the image of $L_n (f),$ where $n$ is the number of columns of $A;$ 
\item for every $n \geq 1,$ the induced morphism $L_n (f) : L_n (R) \to L_n (S)$ is onto;
\item the induced morphism $L_2 (f) : L_2 (R) \to L_2 (S)$ is onto.
\end{enumerate}
\end{theorem}

\begin{proof}
($1$) $\Rightarrow$ ($2$). Factor $A = XPY,$ according to the lemma, so that $XP,$ $P,$ and $PY$ have entries in the image of $f.$ Then
\begin{eqnarray*}
\left[\begin{array}{c|c} XP & \zero \\ -P & PY \end{array} \right] & = & 
\left[\begin{array}{c|c} \zero & XPY \\ -P & PY \end{array} \right] \\
& = & [\zero \; | \; A] \wedge [-P \; | \; PY] = [A].
\end{eqnarray*}
The first equality is obtained by multiplying the bottom ''row'' on the left by $X$ and adding it to the top; the second follows from Proposition~\ref{infimum}; and the last because $[-P \; | \; PY] = 1_n$ (Proposition~\ref{maximum}).
\smallskip

\noindent ($2$) $\Rightarrow$ ($3$). Let $[B \; | \; A] \in L_n (S),$ where $B$ has $k$ columns. By assumption
$$[B, A] = [B' \; | \; A', A'']$$ in $L_{k+n} (S),$ where $A'$ has $k$ columns, $A'$ has $n$ columns, and the entries of $A',$ $A'',$ and $B'$ all lie in the image of $f.$ By Proposition~\ref{systems}.(3),
$[B \; | \; A] =  [B', A' \; | \; A''].$ \smallskip

\noindent ($4$) $\Rightarrow$ ($1$). Let $g,$ $h: S \to T$ be morphisms such that $gf = hf.$ Pick $s \in S;$ we must show that $g(s) = h(s).$ Consider the $1 \times 2$ matrix $A = (1,s).$ By hypothesis, there are matrices $A'$ and $B'$ with entries from the image of $f$ such that 
$$[(1,s)] = [B' \; | \; A'].$$ 
By assumption, the morphisms $g$ and $h$ agree on the entries of $A'$ and $B':$ $g (A') = h(A')$ and $g(B') = h(B').$ By Proposition~\ref{systems}.(2), there are matrices $U,$ $W,$ and $G$ such that 
$$UB' = \zero; \;\;\; UA' = (1,s); \mbox{  and  }  W(1,s) = A' + B'G.$$
Now $A'$ and $G$ are also matrices with $2$ columns so we may express them as $A' = (A'_1, A'_2)$ and $G = (G_1, G_2).$
The equations then become
$$UB' = \zero; \;\;\; U(A'_1, A'_2) = (1,s); \mbox{  and  }  W(1,s) = (A'_1, A'_2) + B'(G_1,G_2).$$
Eliminating $W$ from the third equation yields
$$A'_2 + B'G_2 = Ws = A'_1s + B'G_1s.$$
Apply $p = g - h,$ a function that is not a ring morphism, to both sides of the equation, to obtain
$$g(B')p(G_2) = g(A'_1)p(s) + g(B')p(G_1s).$$
Then multiply on the left by $ g(U)$ and use the fact that $g(U)g(B') = g(UB') = \zero,$ and
$g(U)g(A'_1) = g(UA'_1) = g(1) = 1.$ Whence $p(s) = 0.$ 
\end{proof}

There are morphism $f : R \to S$ that are not epimorphisms, but have the property that $L_1 (f) : L_1 (R) \to L_1 (S)$ is onto. For example, if $f: k \to k'$ is an extension of commutative fields with nontrivial Galois group, 
$\Gal (k'/k) \neq 1.$ Then $f$ is not an epimorphism, but both $L_1 (k)$ and $L_1 (k')$ contain nothing more than the respective maximum and minimum elements, which $L_1 (f)$ respects.

%% file: group.tex
\section{The Grothendieck Group of Finitely Presented Modules}

In this section, we apply the formal moartix calculus to give several presentations of the Grothendieck group $K_0 (\Rmod, \oplus)$ of finitely presented left $R$-modules. One of these presentations is used in the next section as the basis for a homology theory.\bigskip

A left $R$-module ${_R}M$ is {\em finitely presented} if there is an exact sequence, a {\em free presentation} of ${_R}M,$ of the form
\DIAG
{{_R}R^m} \n {\Ear{\grf}} \n {{_R}R^n} \n {\ear} \n {M} \n {\ear} \n {0.} 
\diag
The morphism $\grf : R^m \to R^n$ is given by multiplication on the right by an $m \times n$ matrix $A,$
$$\grf = - \times A.$$ One says that ${_R}M$ is {\em presented} by the matrix $A,$ and writes ${_R}M = M_A.$
Two matrices $A$ and $B$ are {\em equivalent,} denoted $A \sim B,$ if they present isomorphic modules, $M_A \isom M_B.$ The equivalence class of a matrix $A$ is denoted by $\{ A \}.$

Let $\Rmod$ denote the {\em category} of finitely presented modules and define $K_0 (\Rmod, \oplus)$ to be the free group on the symbols $\{ M \},$ $M \in \Rmod,$ modulo the relations 		
$$\{ M \oplus N \} = \{ M \} + \{ N \}.$$ 
It may also be defined as the free group on the equivalence classes $\{ A \}$ of matrices, modulo the relations $$\left\{ \left( \begin{array}{cc}
A & \zero \\ \zero & B \end{array} \right) \right\} = \{ A \} + \{ B \}.$$
This group $K_0 (\Rmod, \oplus)$ is isomorphic to the {\em Grothendieck group} $K_0 (\Ab (R))$ of the {\em free abelian category} $\Ab (R)$ over $R.$ This follows from the fact that the subcategory of projective objects of $\Ab (R)$ is dual to $\Rmod$~\cite{Adel}. 

\begin{theorem} \label{Lickorish} \cite[Theorem 6.1]{Lick}
Two matrices $A$ and $A'$ with entries in the ring $R$ are equivalent $A \sim A'$ if and only if $A'$ may be obtained from $A$ by a sequence of (invertible) operations of the following form:
\begin{enumerate}
\item addition or deletion of an extra row of zeros;
\item replacement of a matrix $C$ by ${\dis \left( \begin{array}{cc}
C & \zero \\ \zero & 1 \end{array} \right),}$ or the reverse;
\item permutation of rows or columns;
\item addition of a left (resp., right) scalar multiple of a row (resp., column) to another row (resp., column).
\end{enumerate}
\end{theorem}	

Theorem~\ref{Lickorish}.(2) implies that for any $n \geq 1,$ the symbol associated to the identity matrix $\{ I_n \} = n \{ 1 \} = 0;$ any identity matrix $I_n$ corresponds to the finitely presented module $0.$ On the other hand, the value of the $m \times k$ zero matrix ${_m}\zero_k$ may be computed using Theorem~\ref{Lickorish}.(1),
$$\{ {_m}\zero_k \} = \{ {_k}\zero_k \} = k\{ 0 \};$$
the $1 \times 1$ matrix $\{ 0 \}$ correspondes to the finitely presented module ${_R}R.$ 

The theorem implies that the association $R \mapsto K_0 (\Rmod, \oplus)$ is functorial, for if $f : R \to S,$ is given, then one may easily verify that whenever two matrices $A$ and $A',$ with entries from $R,$ are equivalent, then so are the matrices $f(A)$ and $f(A')$ with entries from $S.$ The rule $\{ A \} \mapsto \{ f(A) \}$ from $K_0 (\Rmod, \oplus)$
to $K_0 (\Smod, \oplus)$ is therefore well-defined on the generators, and extends linearly to a morphism
$$K_0 (f) : K_0 (\Rmod, \oplus) \to K_0 (\Smod, \oplus)$$ of abelian groups.

\subsection{The Goursat Group}

The {\em Goursat group,} denoted by $G (R),$ is the free group on the elements of $\cup_{n \geq 1}  \; L_n (R),$ modulo the relations:
\begin{enumerate}
\item for every three matrices $A,$ $A'$ and $B$ with the same number of rows, 
$$[B, A \; | \; A'] - [B \; | \; A'] = 	[B, A' \; | \; A] - [B \; | \; A];$$
\item for $[B \; | \; A] \in L_m$ and $[B' \; | \; A'] \in L_n,$
$$\left[ \begin{array}{cc|cc}
B & \zero & A & \zero \\ \zero & B' & \zero & A' \end{array} \right] = [B \; | \; A] + [B' \; | \; A'];$$
\item for every $n \geq 1,$ $0_n = 0.$
\end{enumerate} 
The relations of the Goursat group represent in a formal way the group isomorphism that appears in a celebrated theorem of Goursat~\cite{Gou}. The definition of the Goursat group of $R$ is also functorial: if $f: R \to S$ is a morphism of rings, then the induced morphisms $L_n (f) : L_n (R) \to L_n (S),$ $n \geq 1,$ given by 
$[B \; | \; A] \mapsto [f(B) \; | \; f(A)]$ induce a well-defined function from the generators of $G(R)$ to $G(S),$ which extends linearly to a morphism $G(f) : G(R) \to G(S)$ of Goursat groups.

The $0${\em -Dimensional Goursat Group,} denoted by $G_0 (R),$ is the free group on unary matrix pairs 
$[B \; | \; A],$ the elements of $L_1 (R),$ modulo the relations:
\begin{enumerate}
\item if $A$ and $A'$ are column matrices, and all three matrices $A,$ $A',$ and $B$ have the same number of rows, then $$[B, A \; | \; A'] - [B \; | \; A'] = 	[B, A' \; | \; A] - [B \; | \; A];$$
\item $0_1 = 0.$
\end{enumerate}  
Like the definitions of $K_0 (\Rmod, \oplus)$ and $G(R),$ the definition of $G_0 (R)$ is also functorial. It is clear that the rule $[B \; | \; A] \mapsto [B \; | \; A],$ defined on the generators of $G_0 (R)$ with values in $G(R)$ is well-defined and respects the relations of $G_0 (R).$ It therefore extends linearly to a morphism 
$\iota_R : G_0 (R) \to G (R).$ If $f : R \to S$ is a morphism of rings, then the commutativity of the diagram
\DIAGV{80}
{G_0(R)} \n {\Ear{\iota_R}} \n {G(R)} \nn
{\Sar{G_0(f)}} \n {} \n {\saR{G(f)}} \nn
{G_0(S)} \n {\Ear{\iota_S}} \n {G(S)} 
\diag
is easily established. It shows that the class of morphisms $\iota_R : G_0 (R) \to G (R)$ constitutes a natural transformation $\iota : G_0 \to G$ of functors from the category $\Ring$ of associative rings with identity to the category $\bAb$ of abelian groups.
	
\subsection{The Natural Transformation $\grg : G \to K_0$}	
	
Let the function $$\grg' : \cup_{n \geq 1}  \; L'_n (R) \to K_0 (\Rmod, \oplus)$$ be defined by the rule
$(B \; | \; A) \mapsto \{ B, A \} - \{ B \}.$ In order to show that this function induces as a well-defined function on the generators of $G(R),$ we need the following lemma.

\begin{lemma} \label{elimination}
If $(B \; | \; A) \leq (B' \; | \; A')$ in $L'_n (R),$ then 
$$\left( \begin{array}{ccc}
B & A & \zero \\ \zero & \zero & B' \end{array} \right) \sim
\left( \begin{array}{ccc}
B & A & \zero \\ \zero & A' & B' \end{array} \right).$$
\end{lemma}

\begin{proof} 
We are given $U,$ $V$ and $G$ such that $UB = B'V$ and $UA = A' + B'G.$ One obtains the following sequence (justified below) of equivalences of matrices: 
\begin{eqnarray*}
\left( \begin{array}{ccc}
B & A & \zero \\ \zero & \zero & B' \end{array} \right) & \sim &
\left( \begin{array}{ccc}
B & A & \zero \\ UB & UA & B' \end{array} \right) =
\left( \begin{array}{ccc}
B & A & \zero \\ B'V & UA & B' \end{array} \right)\\
& \sim &
\left( \begin{array}{ccc}
B & A & \zero \\ \zero & UA & B' \end{array} \right) =
\left( \begin{array}{ccc}
B & A & 0 \\ \zero & A' + B'G & B' \end{array} \right) \\
& \sim & \left( \begin{array}{ccc}
B & A & 0 \\ \zero & A' & B' \end{array} \right).
\end{eqnarray*}
The first equivalence is obtained by multiplying the top ''row'' on the left by $U$ and adding it to the second. The other equivalences are obtained by multiplying the third ''column'' on the right by $V$ and $G,$ respectively, then subtracting it from the first, respectively, second.
\end{proof}

The argument in the proof of Lemma~\ref{elimination} was already used in the proof of Theorem~\ref{modular}.

\begin{theorem} \label{gamma}
The rule $\grg' : \cup_{n \geq 1}  \; L'_n (R) \to K_0 (\Rmod, \oplus)$ induces a natural morphism
$\grg_R : G(R) \to K_0 (\Rmod, \oplus)$ of abelian groups, whose values on the generators of $G(R)$ are given by
$$\grg_R : [B \; | \; A] \mapsto \{ B, A \} - \{ B \}.$$
\end{theorem}

\begin{proof}
To begin, let us show that $\grg_R$ is well-defined on the generators of $G(R).$ If 
$[B \; | \; A] = [B' \; | \; A']$ in $L_n (R),$ for some $n \geq 1,$ then 
$(B\; | \; A) \leq (B' \; | \; A')$ and $(B' \; | \; A') \leq (B \; | \; A)$ in $L'_n (R).$ One observes the following sequence of equivalences,
\begin{eqnarray*}
\left( \begin{array}{ccc}
B & A & \zero \\ \zero & \zero & B' \end{array} \right) & \sim & 
\left( \begin{array}{ccc}
B & A & \zero \\ \zero & A' & B' \end{array} \right) \\
& \sim & \left( \begin{array}{ccc}
B' & A' & \zero \\ \zero & A & B \end{array} \right) \\
& \sim & \left( \begin{array}{ccc}
B' & A' & \zero \\ \zero & \zero & B \end{array} \right),
\end{eqnarray*}
where the lemma is used to obtain the first and third equivalences, and Theorem~\ref{Lickorish}.(3) to obtain the second. Thus $\{ B, A \} + \{ B' \} = \{ B', A' \} + \{ B \}$ in the Grothendieck group $K_0 (\Rmod, \oplus).$ Whence
$$\grg_R ([B \; | \; A]) = \{ B,A \} - \{ B \} = \{ B', A' \} - \{ B' \} = \grg_R ([B' \; | \; A']).$$ 

Let us now verify that the relations of $G(R)$ are also respected by $\grg_R.$ If $\grg_R$ is applied to the first family of relations
$$[B, A_1 \; | \; A_2] - [B \; | \; A_2] = 	[B, A_2 \; | \; A_1] - [B \; | \; A_1],$$
then
$$\{ B, A_1, A_2 \} - \{ B, A_1 \} - (\{ B, A_2 \} - \{ B \}) =
\{ B, A_2, A_1 \} - \{ B, A_2 \} - (\{ B, A_1 \} - \{ B \}),$$
which holds in $K_0 (\Rmod, \oplus),$ by Theorem~\ref{Lickorish}.(3).

The second family of relations that hold in $G(R)$ are of the form
$$\left[ \begin{array}{cc|cc}
B & \zero & A & \zero \\ \zero & B' & \zero & A' 
\end{array} \right] = [B \; | \; A] + [B' \; | \; A'].$$
If $\grg_R$ is applied, then
$$\left\{ \begin{array}{cccc}
B & \zero & A & \zero \\ \zero & B' & \zero & A' 
\end{array} \right\} - 
\left\{ \begin{array}{cc}
B & \zero \\ \zero & B'  \end{array} \right\} = \{ B,A \} - \{ B \} + \{ B', A' \} - \{ B' \}.$$
But this clearly holds in $K_0 (\Rmod, \oplus),$ because
$$\left\{ \begin{array}{cccc}
B & \zero & A & \zero \\ \zero & B' & \zero & A' 
\end{array} \right\} = \{ B, A \} + \{ B', A' \}  \mbox{  and  } 
\left\{ \begin{array}{cc} B & \zero \\ \zero & B'  \end{array} \right\} = \{ B \} + \{ B' \}.$$

Finally, to see that $\grg_R : 0_n \mapsto 0,$ recall that $0_n = [{_n}\zero_1 \; | \; I_n].$ Thus
$$\grg_R (0_n) = \{ {_n}\zero_1, I_n \} - \{ {_n}\zero_1  \}.$$
Theorem~\ref{Lickorish}.(2) implies that $\{ {_n}\zero_1, I_n \} = \{ 0, 1 \},$ while Theorem~\ref{Lickorish}.(1) shows that 
$\{ {_n}\zero_1  \} = \{ 0 \}.$ But
$$\{ 0, 1 \} = 
\left\{ \begin{array}{cc}
0 & 1 \\ 0 & 0 \end{array} \right\} =
\left\{ \begin{array}{cc}
0 & 0 \\ 0 & 1 \end{array} \right\} = \{ 0 \},$$
by (1), (3), and (2), respectively, of Theorem~\ref{Lickorish}.
\end{proof}

Let us observe that if $A$ is an $m \times n$ matrix, then $\grg_R ([A]) = \{ A \}.$ For,
$$\grg_R  ([A]) = \grg  ([{_m}\zero_1 \; | \; A]) = \{ {_m}\zero_1 , A \} - \{ {_m}\zero_1 \}.$$ But 
$$\{ {_m}\zero_1, A \} = \left\{ \begin{array}{cc} 
0 & \zero \\ {_m}\zero_1 & A \end{array} \right\} = \{ 0 \} + \{ A \} = \{ {_m}\zero_1 \} + \{ A \},$$
by two applications of Theorem~\ref{Lickorish}.(1).

\subsection{The Natural Transformation $\grk : K_0 \to G_0$}

Define the function $\grk'$ on the collection of all matrices with entries from $R$ to $G_0 (R)$ by the rule
\begin{eqnarray*}
A = (A_1, A_2, \ldots, A_n) & \mapsto & \sum_{i=1}^n \; [A_1, \ldots, A_{i-1} | A_i] \\
& = & [A_1] + [A_1 \; | A_2] + \cdots + [A_1, \ldots, A_{n-2} | A_{n-1}] + [A_1, \ldots, A_{n-1} | A_n].
\end{eqnarray*}
In order to show that this function induces as a well-defined rule on the generators of $K_0 (\Rmod, \Ab),$ we need to
verify that it is invariant under the four operations cited in Theorem~\ref{Lickorish}. We shall use the observation that
$$\grk'_R (A) = \grk'_R (A_1, A_2, \ldots, A_n) =  
\grk'_R (A_1, A_2, \ldots, A_{n-1}) + [A_1, A_2, \ldots, A_{n-1} \; | \; A_n].$$

\noindent {\bf (1) Addition or deletion of an extra row of zeros.} Suppose that $A'$ is obtained from $A$ by adding an extra row of zeros. Then, for each $i,$
$$[A'_1, \ldots, A'_{i-1} | A'_i] = 
\left[ \begin{array}{ccc|c} A_1, & \ldots, & A_{i-1} & A_i \\
0, & \zero, & 0 & 0 \end{array} \right] = [A_1, \ldots, A_{i-1} | A_i]$$
in $L_1 (R),$ by the commentary following Lemma~\ref{more rows}. Consequently,
$$\grk'_R (A) = \sum_{i=1}^n \; [A_1, \ldots, A_{i-1} | A_i] = \sum_{i=1}^n \; [A'_1, \ldots, A'_{i-1} | A'_i] = 
\grk'_R (A').$$

\noindent {\bf (2) Replacement of a matrix $C$ by ${\dis \left( \begin{array}{cc}
C & \zero \\ \zero & 1 \end{array} \right),}$ or the reverse.} Simply note that
\begin{eqnarray*}
\grk'_R \left( \begin{array}{cc} C & \zero \\ \zero & 1 \end{array} \right) & = & 
\grk'_R \left( \begin{array}{c} C  \\ \zero  \end{array} \right) + 
\left[ \begin{array}{c|c} C & \zero \\ \zero & 1 \end{array} \right] \\
& = &  \grk'_R (C) + ([C \; | \; \zero] \wedge [\zero \; | \; 1]) \\
& = &  \grk'_R (C) + (1_1 \wedge 0_1) = \grk'_R (C) + 0_1 = \grk'_R (C),
\end{eqnarray*}
which uses the relation $0_1 = 0$ in $G_0 (R).$\smallskip

\noindent {\bf (3) Permutation of rows or columns.} Suppose that $A'$ is obtained from $A$ by permutation of rows. Then $A' = PA$ for some permutation, hence invertible, matrix $P,$ and so
$$\grk'_R (A') = \sum_{i=1}^n \; [PA_1, \ldots, PA_{i-1} | PA_i] = \sum_{i=1}^n \; [A_1, \ldots, A_{i-1} | A_i] = 
\grk'_R (A),$$ by RoD (1) with $U = P.$

Suppose, on the other hand, that $A'$ is obtained from $A$ by permutation of columns. The {\em symmetric group} $S_n$ on $n$ elements is generated by consecutive transpositions, so it suffices to prove the claim in case $A'$ is obtained from $A$ by permuting consecutive columns $j - 1$ and $j,$ where $1 < j \leq n.$ Let us consider a typical summand 
$$[A_1, \ldots, A_{i-1} | A_i]$$ of $\grk'_R (A).$ If $i > j,$ then a transposition of the $j$-th and $(j-1)$-th columns
will not affect the matrix pair, because it is the result of multiplying the left matrix by a permutation matrix $Q$ on the right. But this leaves the matrix pair invariant, by RoD (2) with $V = Q.$ Also, it is immediate that if $i < j-1,$ then transposing the $j$-th and $(j-1)$-th columns has no effect. The remaining two cases are $i = j-1 $ and $i = j,$ so we need to show that
$$[A_1, \ldots, A_{j-2} | A_{j-1}] + [A_1, \ldots, A_{j-1} | A_j] = 
[A_1, \ldots, A_{j-2} | A_j] + [A_1, \ldots, A_{j-2}, A_j | A_{j-1}]$$
holds in $G_0 (R).$ But
$$[A_1, \ldots, A_{j-1} | A_j] - [A_1, \ldots, A_{j-2} | A_j] = 
[A_1, \ldots, A_{j-2}, A_j | A_{j-1}] - [A_1, \ldots, A_{j-2} | A_{j-1}]$$
is an instance of Relation (1) in the definition of $G_0 (R).$ \smallskip

\noindent {\bf (4) Addition of a left (resp., right) scalar multiple of a row (resp., column) to another row (resp., column).} Suppose that $A'$ is obtained from $A$ by adding a left scalar multiple of a row to another row. Then $A' = PA$ for some elementary, hence invertible, matrix $P.$ Then $\grk'_R (A') = \grk'_R (A),$ just as in the case when $A'$ is obtained from $A$ by permutation of rows.

So suppose that $A'$ is obtained from $A$ by adding a right scalar multiple of a column to another column. To prove that $\grk'_R (A') = \grk'_R (A),$ we are free, by ($3$) above to permute the columns of $A'$ and assume, that it is a right scalar multiple $A_jr$ for $j < n$ that is being added to $A_n,$ the last column, to obtain 
$$\grk'_R (A') = [A_1, \ldots, A_{n-1} | A_n + A_jr] + \sum_{i<n} \; [A_1, \ldots, A_{i-1} | A_i].$$ 
All that remains to be shown is that $[A_1, \ldots, A_{n-1} | A_n + A_jr] = [A_1, \ldots, A_{n-1} | A_n],$ which follows from RoD (3) with $G$ the $(n-1) \times 1$ column matrix whose only nonzero entry $r$ is in the $j$-th row.
 
\begin{theorem} \label{kappa}
The rule $\grk' : A \mapsto \sum_{i=1}^n [A_1, \ldots, A_{i-1} | A_i]$ induces a natural morphism \linebreak
$\grk_R : K_0 (\Rmod, \oplus) \to G_0 (R)$ of abelian groups, whose values on the generators of \linebreak $K_0 (\Rmod, \oplus)$ are given by $\grk_R : \{ A \} \mapsto \sum_{i=1}^n [A_1, \ldots, A_{i-1} | A_i].$
\end{theorem}

\begin{proof}
All that remains to be verified is that the relation
$$\left\{ \left( \begin{array}{cc}
A & \zero \\ \zero & B \end{array} \right) \right\} = \{ A \} + \{ B \},$$
which holds in $K_0 (\Rmod, \oplus)$ is respected by $\grk_R.$ But if one applies $\grk_R,$ then
\begin{eqnarray*}
\grk_R : \left\{ \left( \begin{array}{cc}
A & \zero \\ \zero & B \end{array} \right) \right\} & \mapsto & 
\sum_{i=1}^n \left[ \begin{array}{ccc|c} 
A_1, & \ldots, & A_{i-1} & A_i \\ \zero & \ldots & \zero & \zero \end{array} \right]  
+ \sum_{j=1}^k \left[ \begin{array}{cccc|c} A & \zero & \ldots & \zero & \zero \\ 
\zero & B_1, & \ldots, & B_{j-1} & B_j 
\end{array} \right] \\
& = & \sum_{i=1}^n [A_1, \ldots, A_{i-1} \; | \; A_i]  +  
([A \; | \; \zero] \wedge \sum_{j=1}^k [B_1, \ldots, B_{j-1} \; | \;  B_j]) \\
& = & \grk_R (\{ A \}) + \grk_R (\{ B \}).
\end{eqnarray*}
\end{proof}

\subsection{The Isomorphism Theorem}

For every ring $R,$ we have thus obtained a triangle of natural morphisms
\DIAGV{70}
{G_0 (R)} \n {} \n {\Earv{\iota_R}{120}} \n {} \n {G (R)} \nn 
{} \n {\nwaR{\grk_R}} \n {} \n {\swaR{\grg_R}} \nn
{} \n {} \n {K_0 (\Rmod, \oplus)}
\diag
We will prove that each of these morphisms is an isomorphism, by showing that the composition of all three, regardless of where one begins, yields the identity.\bigskip

\noindent (1) Suppose one begins at $K_0 (\Rmod, \oplus).$ The composition of all three morphisms takes a generator $\{ A \}$ to
\begin{eqnarray*}
\{ A \} & \stackrel{\grk_R}{\mapsto} & \sum_{i=1}^n \; [A_1, \ldots, A_{i-1} \; | \; A_i] \\
& \stackrel{\gri_R}{\mapsto} & \sum_{i=1}^n \; [A_1, \ldots, A_{i-1} \; | \; A_i] \\
& \stackrel{\grg_R}{\mapsto} & \sum_{i=1}^n \; (\{ A_1, \ldots, A_{i-1}, A_i \} -  \{ A_1, \ldots, A_{i-1} \}) \\
& = & \{ A_1, \ldots, A_{n-1}, A_n \} = \{ A \}.
\end{eqnarray*}
The composition $\grg_R \circ \gri_R \circ \grk_R$ is therefore the identity on $K_0 (\Rmod, \oplus).$
\bigskip

\noindent (2) Suppose one begins at $G_0 (R).$ The composition of all three morphisms takes a generator $[B \; | \; A]$
(so that $A$ is a column matrix) to
\begin{eqnarray*}
[B \; | \; A] & \stackrel{\gri_R}{\mapsto} & [B \; | \; A] \\
& \stackrel{\grg_R}{\mapsto} & \{ B, A \} - \{ B \} \\
& \stackrel{\grk_R}{\mapsto} & \sum_{j=1}^k [B_1, \ldots, B_{j-1} \; | \; B_j] + [B \; | \; A] -
\sum_{j=1}^k [B_1, \ldots, B_{j-1} \; | \; B_j] \\
& = & [B \; | \; A].
\end{eqnarray*}
The composition $\grk_R \circ \grg_R \circ \gri_R$ is therefore the identity on $G_0 (R).$\bigskip

To prove the same about the Goursat group $G (R),$ we shall need the following lemma.

\begin{lemma}
Let $A,$ $A'$ and $B$ be matrices with the same number of rows. Then
$$[B \; | \; A, A'] = [B, A \; | A'] + [B \; | \; A]$$
holds in $G(R).$ 
\end{lemma}

\begin{proof}
Let us suppose that $A'$ has $n$ columns. An instance of Relation (1) of the Goursat group is given by
$$\left[\begin{array}{ccc|c} B & A & A' & \zero \\ \zero & \zero & I_n & I_n \end{array}\right] -
\left[\begin{array}{c|c} B & \zero \\ \zero & I_n \end{array} \right] =
\left[\begin{array}{cc|cc} B & \zero & A & A' \\ \zero & I_n & \zero & I_n \end{array}\right] -
\left[\begin{array}{c|cc} B & A & A' \\ \zero & \zero & I_n \end{array} \right]$$
holds in $G(R).$ These terms simplify as follows.
$$\left[\begin{array}{ccc|c} B & A & A' & \zero \\ \zero & \zero & I_n & I_n \end{array}\right] =
\left[\begin{array}{ccc|c} B & A & \zero & -A' \\ \zero & \zero & I_n & I_n \end{array}\right] =
[B, A \; | \; -A'] \wedge [I_n \; | \; I_n],$$
where the first equality follows by multiplying the bottom ''row'' by $A'$ and subtracting it from the top ''row.'' Since $[I_n \; | \; I_n] = 1_n,$ this first term of the relation simplifies to 
$[B,A \; | \; -A'] = [B,A \; | A'],$ which follows by RoD (1) and (2) with $U$ and $V$ matrices of the form $-I_m,$ for appropriate $m.$ Then
$$\left[\begin{array}{c|c} B & \zero \\ \zero & I_n \end{array} \right] = 
[B \; | \; \zero] \wedge [\zero \; | \; I_n] = 1_n \wedge 0_n = 0_n = 0$$
and
$$\left[\begin{array}{cc|cc} B & \zero & A & A' \\ \zero & I_n & \zero & I_n \end{array}\right] =
[B \; | \; A,A'] \wedge [I_n \; | \; \zero, I_n] = [B \; | \; A,A'].$$
The last term of the relation simplifies to
$$\left[\begin{array}{c|cc} B & A & A' \\ \zero & \zero & I_n \end{array} \right] = 
\left[\begin{array}{c|cc} B & A & \zero \\ \zero & \zero & I_n \end{array} \right] =
\left[\begin{array}{cc|cc} B & \zero & A & \zero \\ \zero & \zero & \zero & I_n \end{array} \right]
= [B \; | \; A] + [\zero \; | \; I_n],$$
where the first equality follows by multiplying the bottom ''row'' by $A'$ and subtracting it from the top ''row;'' the second equality from Lemma~\ref{more columns}; and the third from the definition of $G (R).$ Since 
$[\zero \; | \; I_n] = 0_n = 0$ in $G(R),$ we see that the above instance of Relation (1) of $G(R)$ simplifies to
$$[B, A \; | \; A'] = [B \; | \; A, A'] - [B \; | \; A].$$
\end{proof}

\begin{theorem} \label{triangle}
All three morphisms $\iota_R,$ $\grg_R,$ and $\grk_R$ are isomorphisms. Moreover, if one begins at any of the three abelian groups, then the composition obtained by going once around the triangle
\DIAGV{70}
{G_0 (R)} \n {} \n {\Earv{\iota_R}{120}} \n {} \n {G (R)} \nn 
{} \n {\nwaR{\grk_R}} \n {} \n {\swaR{\grg_R}} \nn
{} \n {} \n {K_0 (\Rmod, \oplus)}
\diag 
yields the identity morphism.
\end{theorem}

\begin{proof}
All that remains is to verify that the composition $\gri_R \circ \grk_R \circ \grg_R$ is the identity on $G (R).$
Let $[B \; | \; A]$ be a generator of $G (R).$ First, $\grg_R ([B \; | \; A]) = \{ B, A \} - \{ B \}.$ Then 
\begin{eqnarray*}
\grk_R (\{ B, A \} - \{ B \}) & = &  \sum_{j=1}^k \; [B_1, \ldots, B_{j-1} \; | \; B_j] +
\sum_{i=1}^n \; [B, A_1, \ldots, A_{i-1} \; | \; A_i] \\
&  & \;\;\;\;\; - \sum_{j=1}^k \; [B_1, \ldots, B_{j-1} \; | \; B_j] \\
& = & \sum_{i=1}^n [B, A_1, \ldots, A_{i-1} \; | \; A_i].
\end{eqnarray*}
By the lemma, this is equal to the telescoping sum
$$[B \; | \; A_1] + \{ \sum_{i=2}^n \; ([B \; | \; A_1, \ldots, A_{i-1}, \; A_i] - [B \; | \; A_1, \ldots, A_{i-1}]) \}  = [B \; | \; A].$$
\end{proof}

\subsection{The Positive Cone} A {\em pre-ordered abelian group} $(\grG, \leq)$ consists of an abelian group $\grG$ equipped with a pre-order $\leq$ that satisfies 
$$\grG \models \; \forall x,y,z \; [(x \leq y) \to (x+z \leq y+z)].$$
The {\em positive cone} of $(\grG, \leq)$ is the subset $\grG^+ = \{ x \in \grG \; | \; x \geq 0 \};$ it is an additively closed subset and contains $0 \in \grG^+.$ One may also define the pre-order $\leq$ in terms of the positive cone by $$x \leq y \;\; \mbox{if and only if} \;\; y - x \in \grG^+.$$ The two properties of the positive cone ensure that the relation $\leq$ so defined is a pre-order compatible with the abelian group structure on $\grG.$

The Goursat group $G(R)$ may be equipped with a pre-order whose positive cone $G^+(R)$ is given by the subset of elements having the form
$$[B' \; | \; A'] - [B \; | \; A],$$
where $[B \; | \; A] \leq_n [B' \; | \; A']$ in some $L_n (R),$ $n \geq 1.$ That $G^+(R) \subseteq G(R)$ is additively closed is a consequence of the following lemma, whose proof is left to the reader.

\begin{lemma}
If $(B \; | \; A) \leq_m (B' \; | \; A')$ in $L'_m (R)$ and $(C \; | \; D) \leq_n (C' \; | \; D')$ in $L'_n (R),$ then 
$$\left(\begin{array}{cc|cc} B & \zero & A & \zero \\ \zero & C & \zero & D \end{array} \right) \leq_{m+n}
\left(\begin{array}{cc|cc} B' & \zero & A' & \zero \\ \zero & C' & \zero & D' \end{array} \right)$$
in $L'_{m+n}(R).$
\end{lemma}

\noindent The pre-order that arises on $G^+ (R)$ is the least pre-order such that for every $n \geq 1,$ if 
$[B \; | \; A] \leq_n [B' \; | \; A']$ in $L_n (R),$ then $[B \; | \; A] \leq [B' \; | \; A']$ in $G(R).$ The isomorphism $\grg_R: G(R) \to K_0(\Rmod, \oplus)$ induces a pre-order on the Grothendieck group, whose positive cone is described in the following

\begin{theorem} \label{positive cone}
The positive cone of $K_0 (\Rmod, \oplus)$ is the subset $K_0^+ (\Rmod, \oplus)$ whose elements are of the form
$$\{ A \} - \left\{ \begin{array}{cc} A & \zero \\ B & C \end{array} \right\} + \{ C \},$$
where $A$ and $B$ have the same number of columns, and $B$ and $C$ the same number of rows. 
\end{theorem}

\begin{proof}
Suppose that $[B \; | \; A] \leq_n [B' \; | \; A']$ in $L_n (R),$ for some $n.$ Then
\begin{eqnarray*}
\grg_R ([B' \; | \; A'] - [B \; | \; A]) & = & \{ B' , A' \} - \{ B' \} - \{ B, A \} + \{ B \} \\
& = & \{ B', A' \} - \left\{\begin{array}{ccc} B & A & \zero \\ \zero & \zero & B' \end{array} \right\} - \{ B \} \\
& = & \{ B', A' \} - \left\{\begin{array}{ccc} B & A & \zero \\ \zero & A' & B' \end{array} \right\} - \{ B \} \\
& = & \{ B', A' \} - \left\{\begin{array}{ccc} B' & A' & \zero \\ \zero & A & B \end{array} \right\} - \{ B \}.
\end{eqnarray*} 
The first equality follows from the definition of the Grothendieck group; the second from Lemma~\ref{elimination}; and the third by transposition of the rows, and first and third columns, of the middle matrix.

On the other hand, suppose that $A,$ $B$ and $C$ are given as in the statement of the theorem. By Lemma~\ref{more rows},
$$\left[ \begin{array}{c|c} \zero & A \\ C & B \end{array} \right] \leq [A],$$ so that
\begin{eqnarray*} 
[A] - \left[ \begin{array}{c|c} \zero & A \\ C & B \end{array} \right] & \stackrel{\grg_R}{\mapsto} &
\{ A \} - \left\{ \begin{array}{cc} \zero & A \\ C & B \end{array} \right\} + 
\left\{ \begin{array}{c} \zero \\ C \end{array} \right\} \\
& = & \{ A \} - \left\{ \begin{array}{cc} A & \zero \\ B & C \end{array} \right\} + \{ C \}.
\end{eqnarray*}
The equality follows by transposing the columns of the middle matrix, and applying Theorem~\ref{Lickorish}.(1) to the third matrix.
\end{proof}

Let $\grG$ be a pre-ordered abelian group. A $\grG${\em -character} is a morphism 
$\rho : K_0 (\Rmod, \Ab) \to (\grG, \leq)$ of pre-ordered abelian groups. Equivalently, it is a group morphism that respects positive cones. A $\grG$-character may be considered as a map $\rho$ from the set of equivalence classes of matrices to $\grG$ that satisfies:
\begin{enumerate}
\item $\rho \{ I_n \} = 0,$ for every $n \geq 1;$ 
\item ${\dis  \rho \left\{ \begin{array}{cc} A & \zero \\ \zero & B \end{array} \right\} = \rho \{ A \} + \rho \{ B \};}$ and
\item ${\dis  \rho \left\{ \begin{array}{cc} A & \zero \\ B & C \end{array} \right\} \leq \rho \{ A \} + \rho \{ C \}.}$
\end{enumerate}
One may verify that these three properties imply that $\rho (AB) \geq \rho (A)$ and $\rho (AB) \geq \rho (B).$
These maps are closely related to the Sylvester rank functions of Schofield~\cite[pp.\ 96-97]{Scho} and, when $\grG = \bZ$ is the totally ordered group of integers, correspond to the characters studied by Crawley-Boevey~\cite[\S 5]{CB}. If the elements of $K_0 (\Rmod, \oplus)$ are represented as linear combinations of symbols $\{ M \},$ where $M \in \Rmod,$ then the three properties of a $\grG$-character may be rephrased as: 
\begin{enumerate}
\item $\rho \{ 0 \} = 0;$ 
\item $\rho (M \oplus N) = \rho (M) + \rho (N);$ and
\item if $M \to N \to K \to 0$ is an exact sequence in $\Rmod,$ then $\rho (N) \leq \rho (M) + \rho (K).$
\end{enumerate}
Since $\Rmod$ is closed under cokernels, the last property is equivalent to the condition that for every morphism 
$f: M \to N$ in $\Rmod,$ $\rho (N) \leq \rho (M) + \rho (\Coker f).$

%% file: homo.tex
\section{Homology}

In this section, a complex is introduced whose group of $n$-chains is the free abelian group on the $(n+1)$-ary matrix pairs, and whose $0$-dimensional homology group is naturally isomorphic to $G_0 (R),$ the $0$-dimensional Goursat group of $R.$

\subsection{The Face Operations} In order to define the boundary map of the complex, let us introduce the faces of an $(n+1)$-ary matrix pair $[B \; | A_0, A_1, \ldots, A_n \; ].$ We shall denote by $\hat{A}_i$ the matrix with $n$ columns obtained by removing from $A$ the column $A_i,$ 
$$\hat{A}_i := (A_0, A_1, \ldots, A_{i-1}, A_{i+1}, \ldots, A_{n+1}).$$

\begin{proposition} \label{faces}
Let $(B \; | \; A) \leq (B' \; | \; C)$ in $L'_{n+1} (R).$ Then, for every $i,$ $0 \leq i \leq n,$
\begin{enumerate}
\item $(B \; | \; \hat{A}_i) \leq (B' \; | \; \hat{C}_i)$ and
\item $(B, A_i \; | \; \hat{A}_i) \leq (B', C_i \; | \; \hat{C}_i)$
\end{enumerate}
in $L'_n (R).$
\end{proposition}

\begin{proof}
We are given matrices $U,$ $V$ and $G$ satisfying 
$$UB = BV \;\;\; \mbox{and} \;\;\; UA = C + B'G.$$
Equating the $i$-columns of the second equation gives the equation $UA_i = C_i + B'G_i;$ equating the other columns gives $U\hat{A}_i = \hat{C}_i + B'\hat{G}_i.$ From these equation we derive that:
\begin{enumerate}
\item $UB = BV$ and $U\hat{A}_i = \hat{C}_i + B'\hat{G}_i,$ which implies that 
$(B \; | \; \hat{A}_i) \leq (B' \; | \; \hat{C}_i);$ and \bigskip
\item ${\dis U(B, A_i) = (B', C_i) \left( \begin{array}{cc} V & \hat{G}_i \\ \zero & I_n \end{array} \right)}$ and
${\dis U\hat{A}_i = \hat{C}_i + (B', C_i) \left( \begin{array}{c} \hat{G}_i \\ \zero \end{array} \right),}$ which implies that $(B, A_i \; | \; \hat{A}_i) \leq (B', C_i \; | \; \hat{C}_i).$
\end{enumerate}
\end{proof}

We think of an $(n+1)$-matrix pair $[B \; | \; A]$ as a cube, and define its faces using Proposition~\ref{faces}. For each $i$ such that $0 \leq i \leq n,$ let the {\em top} $i${\em -face} of $[B \; | \; A]$ be given by
$$E_i [B \; | \; A] := [B, A_i \; | \; \hat{A}_i];$$
and the {\em bottom} $i${\em -face} by
$$N_i [B \; | \; A] := [B \; | \; \hat{A}_i].$$
By Proposition~\ref{faces}, these two face operators are well-defined morphisms of linear orders:
$$E_i, \; N_i : L_{n+1}(R) \to L_n (R).$$
The bottom $i$-face operator $N_i$ respects the minimum element,
$$N_i (0_{n+1}) = N_i [1 \; | \; {_1}\zero_{n+1}] = [1 \; | \; {_1}\zero_n] = 0_n,$$
and preserves the infimum operation,
\begin{eqnarray*}
N_i ([B\; | \; A] \wedge [B' \; | \; C]) & = & N_i 
\left[ \begin{array}{cc|c} B & \zero & A \\ \zero & B' & C  \end{array} \right] \\
& = & \left[ \begin{array}{cc|c} B & \zero & \hat{A}_i \\ \zero & B' & \hat{C}_i  \end{array} \right] \\
& = & N_i [B \; | \; A] \wedge N_i [B' \; | \; C].
\end{eqnarray*}
Dually, the top $i$-face operator $E_i,$ respects the maximum element, and preserves the supremum operation. This will follow from the next proposition, which relates the top and bottom $i$-faces of an $(n+1)$-ary matrix cube in a manner similar to the way in which the infimum and supremum operations are related in $L_n (R).$ To state the proposition, we denote by $E^*_i$ and $N^*_i$ the respective $i$-face operators on the opposite partial order $L_{n+1} (R^{\op}).$

\begin{proposition} \label{dual faces}
Let $[B \; | \; A]$ be an $(n+1)$-ary matrix pair in $L_{n+1}(R).$ Then
$$E^*_i [B \; A]^* = (N_i [B \; | \; A])^*$$
holds in $L_n (R^{\op}).$
\end{proposition}

\begin{proof}
By definition, we have that
$$E^*_i [B \; | \; A]^* = E^*_i \left[ \begin{array}{c|c} B^{tr} & \zero \\ A^{tr} & I_{n+1}  \end{array} \right]
= \left[ \begin{array}{cc|c} B^{tr} & \zero & \zero \\ A^{tr} & \bei & (\widehat{I_{n+1}})_i  \end{array} \right],$$
where $\bei$ is the $(n + 1) \times 1$ column matrix with the unique nonzero entry $1$ in the $i$-row, the indexing begining at $0.$ The $i$-row of $A^{\tr}$ is given by $(A_i)^{\tr}.$ If that row is brought down to the bottom of the matrix pair, we obtain
$$\left[ \begin{array}{cc|c} B^{tr} & \zero & \zero \\ (\widehat{A_i})^{tr} & \zero & I_n \\
(A_i)^{\tr} & 1 & \zero \end{array} \right] =
\left[ \begin{array}{cc|c} B^{tr} & \zero & \zero \\ (\widehat{A_i})^{tr} & \zero & I_n \\
\zero & 1 & \zero \end{array} \right].$$
The equation follows by performing column operations on the left matrix to eliminate the nonzero entries of $(A_i)^{\tr}.$ But this last $n$-ary matrix pair is just 
$$(N_i [B \; | \; A])^* \wedge [1 \; | \; {_1}\zero_n] = (N_i [B \; | \; A])^* \wedge 1_n = (N_i [B \; | \; A])^*.$$ 
\end{proof}

The relations between the various $i$-face and $j$-face operators is described next. They are based on the observation that if $A$ has $n+1$ columns and $0 \leq i < j \leq n,$ then the $j$-column of $A$ becomes the $(j-1)$-column of
$\hat{A}_i,$ $$A_j = (\hat{A}_i)_{j-1}.$$ Thus the matrix obtained from $A$ by removing the $i$ and $j$-columns may be represented as $$\widehat{(\hat{A}_j)}_i = \widehat{(\hat{A}_i)}_{j-1}.$$
 
\begin{proposition} \label{face relations}
For $0 \leq i < j \leq n,$ the face operators are related according to the equations
\begin{enumerate}
\item $N_i N_j = N_{j-1} N_i$ and $E_i E_j = E_{j-1} E_i;$ and
\item $N_i E_j = E_{j-1} N_i$ and $E_i N_j = N_{j-1} E_i.$ 
\end{enumerate}
\end{proposition} 

\begin{proof}
We only prove the first of each pair; the second follows from the first by duality. For example, once $N_i N_j = N_{j-1} N_i$ is proved for $R,$ it is also true for the opposite ring $R^{\op}.$ Proposition~\ref{dual faces} then implies $E_i E_j = E_{j-1} E_i$ for $R.$

To prove $N_i N_j = N_{j-1} N_i,$ simply note that if $[B \; | \; A]$ is an $(n+1)$-ary matrix pair, then
$$N_i N_j [B \; | \; A] = N_i [B \; | \; \hat{A}_j] = [B \; | \; \widehat{(\hat{A}_j)}_i] = [B \; | \; \widehat{(\hat{A}_i)}_{j-1}] = N_{j-1} [B \; | \; \hat{A}_i] =  N_{j-1} N_i [B \; | \; A].$$
Similarly, $N_i E_j [B \; | \; A] = N_i [B, A_j \; | \; \hat{A_j}] = [B, A_j \; | \; \widehat{(\hat{A_j})}_i] = 
[B, (\hat{A}_i)_{j-1} \; | \; \widehat{(\hat{A}_i)}_{j-1}],$ while
$$E_{j-1} N_i [B \; | \; A] = E_{j-1} [B \; | \; \hat{A}_i] = [B, (\hat{A}_i)_{j-1} \; | \; \widehat{(\hat{A}_i)}_{j-1}] = N_i E_j [B \; | \; A].$$
\end{proof}

\subsection{The Complex of Matrix Pairs}

Let us define a complex $Q_* (R)$ of abelian groups, indexed by the integers $\bZ$ so that
\begin{enumerate}
\item if $n \geq 0,$ then $Q_n (R)$ is the free abelian group on the $(n+1)$-ary matrix pairs $[B\; | \; A];$
\item if $n = -1,$ then $Q_{-1}(R) = \bZ;$ and 
\item if $n < -1,$ then $Q_n (R) = 0.$
\end{enumerate}
The boundary map $\partial_n : Q_n (R) \to Q_{n_1}(R)$ is defined so that: 
\begin{enumerate}
\item if $n > 0,$ then $\partial_n$ is defined on an $(n+1)$-ary matrix pair by the rule
$$[B \; | \; A] \mapsto \sum_{i=0}^n \; (-1)^i (E_i[B \; | \; A] - N_i[B \; | \; A]),$$
and extended linearly;
\item if $n = 0,$ then $\partial_0 : Q_0 (R) \to \bZ$ is the augmentation map $\epsilon,$ defined on a unary matrix pair by the rule $\epsilon: [B \; | \; A] \mapsto 1,$ and extended linearly; and
\item if $n \leq -1,$ then, of course, $\partial_n = 0.$
\end{enumerate}
If the definition of the $i$-face operators $E_i$ and $N_i$ is extended linearly, to morphisms $E_i$ and $N_i$ from $Q_n (R) \to Q_{n-1},$ then, for $n > 0,$ we may express the boundary operator as
$$\partial_n = \sum_{i=0}^n \; (-1)^i(E_i - N_i).$$
It is routine (cf.\ p.13 of~\cite{Mass}) to verify that $Q_* (R)$ is indeed a complex.

\begin{proposition}
For every integer $n,$ the composition $\partial^n \circ \partial^{n+1} = 0.$
\end{proposition}

The complex $Q_* (R)$ is called the {\em complex of matrix pairs} of $R.$ If $f : R \to S$ is a morphism of rings, then for every $n \geq 0,$ the morphism $L_{n+1} (f) : L_{n+1} (R) \to L_{n+1} (S)$ of partially ordered sets may be extended linearly, to obtain a morphism $Q_n (f) : Q_n (R) \to Q_n (S)$ of abelian groups. If the morphism 
$Q_{-1} (f) : \bZ \to \bZ$ is defined to the identity, and $Q_m (f) = 0$ for $m < -1,$ then we obtain the following,
which requires a routine verification (cf.\ \S II.3 of~\cite{Mass}).

\begin{proposition} 
For every morphism $f : R \to S$ of rings, the family $Q_* (f) : Q_* (R) \to Q_* (S)$ is a morphism of complexes.
\end{proposition}

The anti-isomorphism $[B \; | \; A] \mapsto [B \; | \; A]^*$ from $L_{n+1} (R)$ to $L_{n+1} (R^{\op})$ may be extended linearly, to obtain an isomorphism $p_n : Q_n (R) \to Q_n (R^{\op})$ of abelian groups, for $n \geq 0.$ Proposition~\ref{dual faces} implies that these morphisms $p_n$ satisfy the equations
$$p_{n-1}N_i = E^*_i p_n \;\; \mbox{and} \;\; p_{n-1}E_i = N^*_i p_n.$$ 
Define 
$P_n : Q_n (R) \to Q_n (R^{\op})$ to be the map 
$$P_n := \left\{ \begin{array}{ll} (-1)^{n+1}p_n & \;\; \mbox{if $n \geq 0;$} \\
-1_{\bZ} & \;\; \mbox{if $n = -1;$ and} \\
0 & \;\; \mbox{if $n < -1.$}
\end{array} \right.$$
It is routine to verify that $P_* : Q_* (R) \to Q_* (R^{\op})$ is a morphism of complexes. The inverse of $P_*$ is defined similarly, and yields the following.

\begin{proposition} \label{homo dual}
The morphism $P_* : Q_* (R) \to Q_* (R^{\op})$ of complexes is an isomorphism.
\end{proposition}

\subsection{$0$-Dimensional Homology}

The $0$-dimensional homology of $R$ is defined to be the homology of the complex $Q_* (R)$ at $Q_0 (R),$ 
\DIAG
{\cdots} \n {\Ear{\partial_2}} \n {Q_1 (R)} \n {\Ear{\partial_1}} \n {Q_0 (R)} \n {\Ear{\epsilon}} \n {\bZ} \n {\ear} \n {0.}
\diag
The kernel of $\epsilon$ is the subgroup $B_0 (R)$ of $0${\em -boundaries;} the image of $\partial_1$ the subgroup 
$Z_0 (R)$ of $0${\em -cycles.} Because $\epsilon \circ \partial_1 = 0,$ the inclusion $Z_0 (R) \subseteq B_0 (R)$ holds, and the quotient group $H_0 (R) := B_0 (R)/Z_0 (R)$ is the $0${\em -dimensional homology} of $R.$ All of these notions are functorial in $R.$

\begin{theorem} \label{0 homo}
The morphism $\grl_R : G_0(R) \to H_0(R)$ induced by the rule 
$$\grl'_R : [B \; | \; A] \mapsto [B \; | \; A] - 0_1 \mod B_0 (R)$$ is a natural isomorphism.
\end{theorem}

\begin{proof}
Let us verify that the rule $\grl'_R$ so defined respects the relations that define $G_0 (R).$ The family of relations given by $[B, A \; | \; A'] - [B \; | \; A']	- ([B, A' \; | \; A] - [B \; | \; A])$ is mapped by $\grl'_R$ to
\begin{eqnarray*}
& & ([B, A \; | \; A'] - 0_1) - ([B \; | \; A'] - 0_1	- ([B, A' \; | \; A] - 0_1) + ([B \; | \; A] - 0_1) \\
& = & [B, A \; | \; A'] - [B \; | \; A']	- ([B, A' \; | \; A] - [B \; | \; A]) \\
& = & E_0 [B \; | \; A, A'] - N_0 [B \; | \; A, A'] - (E_1 [B \; | \; A, A'] - N_1 [B \; | \; A, A']),
\end{eqnarray*}
which is the boundary $\partial_1 [B \; | \; A, A'].$ Also, the element $0_1$ maps to $0_1 - 0_1 = 0.$

Let us define a map $\grd_R : H_0 (R) \to G_0 (R)$ in the other direction, which will be the inverse of $\grl_R.$ To do so, note first that the group $Z_0 (R)$ of $0$-cycles is free on the generators $[B \; | \; A] - 0_1,$ where 
$[B \; | \; A] \in L_1 (R)$ is not the minimum element. This is seen by considering the splitting of the short exact sequence
\DIAG
{0} \n {\ear} \n {Z_0 (R)} \n {\ear} \n {Q_0 (R)} \n {\Ear{\epsilon}} \n {\bZ} \n {\ear} \n {0}
\diag 
that corresponds to the section $s : \bZ \to Q_0 (R)$ determined by $1 \mapsto 0_1.$ Define 
$\grd'_R : Z_0 (R) \to G_0 (R)$ by the rule $\grd'_R : [B \; | \; A] - 0_1 \mapsto [B \; | \; A].$ To see that $\grd'_R$ induces a morphism $\grd_R : H_0 (R) \to G_0 (R),$ it suffices to check that it annihilates every $0$-boundary. But if $[B \; | \; A, A']$ is a binary matrix pair, then
$$\grd'_R \circ \partial_1 [B \; | \; A, A'] = 
[B, A \; | \; A'] - [B \; | \; A']	- ([B, A' \; | \; A] - [B \; | \; A])$$
is an instance of Relation (1) of $G_0 (R).$ 

The two morphisms $\grl_R : G_0 (R) \to H_0 (R)$ and $\grd_R : H_0 (R) \to G_0 (R)$ are clearly mutual inverses, for if
we compose them in either order, the generators remain fixed.
\end{proof}

The isomorphisms of Theorem~\ref{0 homo} and Proposition~\ref{homo dual} may be composed to yield an isomorphism
\DIAG
{G_0(R)} \n {\Ear{\grl_R}} \n {H_0 (R)} \n {\Ear{P_0}} \n {H_0 (R^{\op})} \n {\Ear{\grl^{-1}_{R^{\op}}}} \n {G_0(R^{\op})}
\diag  
of $0$-dimensional Goursat groups. Composing further with the isomorphism $\grk$ of Theorem~\ref{triangle} yields an isomorphism between the $K_0 (\Rmod, \oplus)$ and the Grothendieck group $K_0 (\modR, \oplus)$ of the category $\modR$ of finitely presented right $R$-modules. One may check that if $A$ is a matrix with $n$ columns, then this isomorphism sends the class $\{ A \}$ in $K_0 (\Rmod, \oplus)$ to the element $n \{ 0 \} - \{ A^{\tr} \}$ in $K_0 (\modR, \oplus).$

\subsection{Degeneracy} An $(n+1)$-ary matrix pair $[B \; | \; A],$ $n > 0,$ is {\em degenerate} if for some $i,$ 
$0 \leq i \leq n+1,$ $$E_i [B \; | \; A] = N_i [B \; | \; A].$$
In order to show that the degenerate matrix pairs generate a subcomplex that does not change the $0$-dimensional homology, we shall need the following lemma.

\begin{lemma}
Suppose that $A$ is an $m \times n$ matrix, and $A'$ an $m \times n'$ matrix. If 
$(B, A \; | \; A') \leq (B \; | \; A')$ in $L'_{n'}(R),$ then $(B, A' \; | \; A) \leq (B \; | \; A)$ in $L'_{n}(R).$
\end{lemma}

\begin{proof}
We are given $U,$ $V = (V', V'')$ and $G$ such that $U(B,A) = B(V',V'')$ and $UA' = A' + BG.$ From that, we get
$$(I_m - U)(B, A') = (B - UB, A' - UA') = (B - BV', - BG) = B (I - V', -G)$$
and $(I_m - U)A = A - UA = A - BV'' = A + B(-V'').$ 
\end{proof}

\begin{proposition} \label{2-degenerate}
If $[B \; | \; A, A'] \in L_2(R),$ then $E_0 [B \; | \; A, A'] = N_0 [B \; | \; A, A']$ if and only if
$E_1 [B \; | \; A, A'] = N_1 [B \; | \; A, A'].$
\end{proposition}

\begin{proof}
Suppose that $$E_0 [B \; | \; A, A'] = [B, A \; | \; A'] = [B \; | \; A'] = N_0 [B \; | \; A, A'].$$ Then
$(B, A \; | \; A') \leq (B \; | \; A'),$ implies, by the lemma, that $(B, A' \; | \; A) \leq (B \; | \; A).$ On the other hand, Lemma~\ref{more columns} implies that $(B \; | \; A) \leq (B, A' \; | \; A).$ Thus
$$E_1 [B \; | \; A, A'] = [B, A' \; | \; A] = [B \; | \; A] = N_1 [B \; | \; A, A'].$$
These direction of the equivalence holds, in particular, in $L_2 (R^{\op}).$ The converse now follows by an application of Proposition~\ref{dual faces}.
\end{proof}

The proposition implies that if $[B \; | \; A, A'] \in L_2 (R)$ is degenerate, then $\partial_1 [B \; | \; A, A'] = 0.$
For $n > 0,$ let $M_n (R) \subseteq Q_n (R)$ be the subgroup generated by degenerate matrix pairs. If $n \leq 0,$ let
$M_n (R) \subseteq Q_n (R)$ be the zero subgroup.\bigskip

\begin{proposition}
The family of subgroups $M_n (R) \subseteq Q_n (R),$ $n \in \bZ,$ consitutes a subcomplex $M_* (R) \subseteq Q_* (R).$
\end{proposition}  

\begin{proof}
Suppose that $n > 1$ and $[B \; | \; A] \in L_{n+1}$ is degenerate. Then $E_j[B \; | \; A] = N_j[B \; | \; A]$ for some $j,$ $0 \leq j \leq n+1,$ and it is easy to check that 
$$\partial_n [B \; | \; A] = \sum_{i \neq j} \; (-1)^i (E_i [B \; | \; A] - N_i [B \; | \; A])$$
is a linear combination of degenerate $(n-1)$-ary matrix pairs. Thus $\partial_n (M_n (R)) \subseteq M_{n-1}(R)$ for all $n > 1.$ For $n = 1,$ $\partial_1 (M_1 (R)) = 0 = M_0 (R)$ is a consequence of Proposition~\ref{2-degenerate}. 
\end{proof}

The complex in which we are interested is given by $C_* (R) := Q_* (R)/M_* (R).$ It is called the {\em complex of nondegenerate matrix pairs} of $R.$ If $f: R \to S$ is a morphism of rings, then for every $n,$ 
$Q_n (f) : M_n (R) \subseteq M_n (S),$ which induces a morphism of complexes $C_* (f) : C_* (R) \to C_* (S).$ If 
$f : R \to S$ is an epimorphism, then $Q_* (f)$ is an epimorphism of complexes, as is the induced $C_* (f).$ 

For every $n,$ denote by $B_n (R) \subseteq C_n (R)$ the subgroup of $n$-boundaries, and by $Z_n (R)$ the subgroup of $n$-cycles. Note that these definitions generalize the above definitions of $B_0 (R)$ and $Z_0 (R),$ because 
$C_0 (R) = Q_0 (R)$ and $\partial_1 (M_1 (R)) = 0.$ In particular, if we define the $n${\em -dimensional homology} of $R$ to be the homology $Z_n (R)/B_n (R)$ of $C_* (R)$ at $C_n (R),$ then this coincides with the earlier definition of $H_0 (R)$ for $n = 0.$

If $f : R \to S$ is an epimorphism of rings, then the morphism $Q_* (f) : Q_* (R) \to Q_* (S)$ is an epimorphism of complexes, so the induced morphism $C_* (f) : C_* (R) \to C_* (S)$ is also an epimorphism. Each of the abelian groups 
$Q_n (S)/M_n (S)$ is free, so that a long exact sequence of homology arises.

\subsection{$1$-Dimensional Homology of a Field} Let $k$ be a, not necessarily commutative, field. Corollary~\ref{von Neumann} implies that every matrix pair over $k$ is equivalent to a system. The modular lattice $L_1(k)$ is trivial; it consists of exactly two unary matrix pairs $[0] = 1_1$ and $[1] = 0_1,$ the maximum and minimum elements, respectively. It follows that
$$Q_0(k) = C_0(k) = \bZ 0_1 \oplus \bZ 1_1,$$
and that the group of 0-cycles is given by $Z_0(k) = \bZ(1_1 - 0_1).$ The group of $0$-boundaries is trivial, 
$B_0 (k) = 0,$ because the nondegenerate elements of $L_2 (k)$ have the form $[1, r],$ with $r \neq 0,$ and
$$\partial_1 ([1, r]) = ([1 \; | \; r] - [r]) - ([r \; | \; 1]  - [1]) = (1_1 - 0_1) - (1_1 - 0_1) = 0.$$
Since $\partial_1 = 0$ every $1$-chain is a $1$-cycle, $C_1 (k) = Z_1 (k),$ so that $Z_1 (k)$ is the free abelian group on the elements $[1 \; | \; r],$ where $r \in k^{\times},$ the multiplicative group of nonzero elements of $k.$

There are two families of nondegenerate generators of $C_2 (k),$ given by 
$$[1,r,s] \mbox{ and } \left[ \begin{array}{ccc}
1 & 0 & r \\
0 & 1 & s
\end{array} \right].$$
The boundary map of the first family may be computed as
\begin{eqnarray*}
\partial_2 [1,s,r] & = & ([1 \; | \; s, r] - [s,r]) - ([s \; | \; 1, r] - [1,r]) + ([r \; | \; 1, s] - [1,s]) \\
& = & (1_2 - [1, s^{-1}r]) - (1_2 - [1,r]) + (1_2 - [1,s]) \\
& = & [1,r] - [1,s] - [1, s^{-1}r],
\end{eqnarray*}
because $1_2$ is a degenerate tertiary matrix pair. Similarly,
\begin{eqnarray*}
\partial_2 \left[ \begin{array}{ccc}
1 & 0 & r \\
0 & 1 & s
\end{array} \right] & = &
(\left[ \begin{array}{c|cc}
1 & 0 & r \\
0 & 1 & s
\end{array} \right] - \left[ \begin{array}{cc}
0 & r \\
1 & s
\end{array} \right]) \\
& - & (\left[ \begin{array}{c|cc}
0 & 1 & r \\
1 & 0 & s
\end{array} \right] - \left[ \begin{array}{cc}
1 & r \\
0 & s
\end{array} \right]) \\
& + & (\left[ \begin{array}{c|cc}
r & 1 & 0 \\
s & 0 & 1
\end{array} \right] - \left[ \begin{array}{cc}
1 & 0 \\
0 & 1
\end{array} \right]). 
\end{eqnarray*}
All three systems that appear have a matrix of coefficients row equivalent to $I_2,$ and are therefore equivalent to the degenerate binary matrix pair $[I_2] = 0_2.$ To simplify the other terms, just note that
$$\left[ \begin{array}{c|cc}
1 & 0 & r \\
0 & 1 & s
\end{array} \right] = [1 \; | \; 0,r] \wedge [1,s] = 1_2 \wedge [1,s] = [1,s];$$
$$\left[ \begin{array}{c|cc}
0 & 1 & r \\
1 & 0 & s
\end{array} \right] = [1,r] \wedge [1 \; | \; 0,s] = [1,r] \wedge 1_2 = [1,r];$$
and, the most interesting case,
$$\left[ \begin{array}{c|cc}
r & 1 & 0 \\
s & 0 & 1
\end{array} \right] = \left[ \begin{array}{c|cc}
r & 1 & 0 \\
0 & -sr^{-1} & 1
\end{array} \right] = [r \; | \; 1,0] \wedge [-sr^{-1}, 1] = 1_2 \wedge [1, - rs^{-1}] = [1, -rs^{-1}].$$
The group $B_1 (k) \subseteq Z_1 (k)$ of $1$-boundaries is therefore generated by the two families
\begin{equation} \label{1-boundary A}
[1,r] - [1,s] - [1, s^{-1}r]
\end{equation}
and 
\begin{equation} \label{1-boundary B}
[1,s] - [1,r] + [1, -rs^{-1}],
\end{equation}
as $r$ and $s$ vary over the multiplicative group $k^{\times}.$

\begin{theorem} \label{1-dim relations}
If $k$ is a field, then $H_1(k)$ is the abelian group generated by the symbols $[1,r],$ $r \in k^{\times}$ modulo the relations $[1,r] - [1,s] - [1, rs^{-1}]$ and $[1,-1].$ 
\end{theorem}

\begin{proof}
First, we show that these two relations hold in $H_1 (k).$ Letting $r = s = 1$ in both relations yields 
$[1,1] = [1, -1]$ in $H_1 (k);$ then, letting $r = 1$ in both relations yields $[1, \pm s^{-1}] = - [1,s]$ in $H_1 (k);$ and, finally, letting $s = t^{-1}$ in the first relation yields
$$[1,r] + [1,t] = [1,tr].$$
Thus $[1, tr] = [1, rt],$ which applied to the first relation, yields
$$[1,r] - [1,s] = [1, s^{-1}r] = [1, rs^{-1}].$$

Suppose, on the other hand, that we are given the relations $[1,r] - [1,s] - [1, rs^{-1}]$ and $[1,-1].$ Letting $r = s$ yields $[1,1] = 0;$ then, letting $r = 1$ yields $[1, s^{-1}] = -[1,s];$ and finally, letting $s = t^{-1}$ yields
$$[1,r] + [1,t] = [1,r] - [1,t^{-1}] = [1,rt].$$
Whence $[1,r] - [1,s] = [1, rs^{-1}] = [1, s^{-1}r]$ and 
$[1,s] - [1,r] = [1,rs^{-1}] = [1, rs^{-1}] - [1,-1] = [1, -rs^{-1}].$
\end{proof}

Recall~\cite[\S 2.2]{Rosen} that the first $K$-group of a field $k$ is the abelianization of
the multiplicative group, $$K_1 (k) = (k^{\times})_{\ab}.$$
A standard group-theoretic argument shows that if $\grG$ is a group, then its abelianization $\grG_{\ab}$ is isomorphic to the free abelian group on the symbols $x_g,$ $g \in \grG,$ modulo the family of relations 
$$x_g - x_h - x_{gh^{-1}},$$
as $g$ and $h$ vary over $\grG.$ Theorem~\ref{1-dim relations} thus implies the following.

\begin{corollary} \label{1-dim homo} 
$H_1 (k) \isom (k^{\times})_{\ab}/\{ \pm 1 \}.$
\end{corollary}

%% file: appendix.tex
\section{The Model Theory of Modules}

In this final section, we consider several characterizations of the pre-order $\leq_n$ on $L'_n (R).$ These various, but equivalent, formulations may seem more concrete, and serve as a historical reference for the origin of the pre-order $\leq_n.$

\subsection{A Completeness Theorem} The {\em language} (cf.~\cite[\S 1.1]{Mike's book}) for left $R$-modules is $\Lang (R) = (+,-,0,r)_{r \in R}.$ It contains the language $(+,-,0)$ for abelian groups together with unary function symbols $r,$ one for every element $r \in R,$ and denoted the same way. The symbols from the abelian group language are intended to interpret the underlying abelian group structure of a left $R$-module ${_R}M,$ while the unary function symbol $r$ is intended to interpret the action of the corresponding ring element $r \in R$ on $M.$

The {\em atomic formulae} in $\Lang (R)$ are of the form
\begin{equation} \label{atom}
b_1w_1 \pm b_2w_2 \pm \cdots \pm b_mw_m \doteq a_1v_1 \pm a_2v_2 \pm \cdots \pm a_nv_n,
\end{equation}
where the $v_i$ and $w_j$ are variables, and $a_i$ and $b_j$ elements from $R$ acting as scalars on the left.
These formulae are nothing more than linear equations.

The standard axioms for a left $R$-module are expressible in $\Lang (R).$ In the language $\Lang (R),$ it is not possible to quantify over the ring $R,$ so that this collection of axioms, denoted by $T(R),$ is usually infinite. For example, $T(R)$ contains the axiom schema: for every $r \in R,$ $$(\forall v,w) \; r(v + w) \doteq rv + rw.$$
Relative to the axioms $T(R),$ we may rewrite the atomic formula~(\ref{atom}) as 
$$b_1w_1 + b_2w_2 + \cdots + b_mw_m \doteq a_1v_1 + a_2v_2 + \cdots + a_nv_n.$$
A {\em system} of linear equations is expressible as a finite conjunction of linear equations, 
$$B\bw \doteq A\bv,$$
where $A$ is an $m \times n$ matrix, $B$ an $m \times k$ matrix and $\bv = (v_i)_{i=1}^n$ and $\bw = (w_j)_{j=1}^k$ {\em column vectors} of $n,$ (respectively, $k$) variables.

A {\em positive-primitive formula} is an existentially quantified system of linear equations:
$$\exists \bw \; (B\bw \doteq A\bv).$$
There are two extreme cases of a positive-primitive formula;
\begin{enumerate}
\item if $B = \zero,$ then the positive-primitive formula is equivalent, relative to $T(R),$ to a {\em quantifier-free}
formula, the homogeneous system of linear equations $A\bv \doteq \zero;$ 
\item if $A = I_n,$ the $n \times n$ identity matrix, then the positive-primitive formula is equivalent, relative to $T(R),$ to the {\em divisibility condition} $$B | \bv := \exists \bw \; (B\bw \doteq \bv).$$
Using this notation, we may express a general pp-formula in the free variables $\bv$ as 
$$(B \; | \; A)(\bv) := B | A\bv.$$
\end{enumerate}

\noindent The Rules of Divisibility correspond to three properties of positive-primitive formulae:
\begin{enumerate}
\item $T(R) \vdash \forall \bv \; [(B \; | \; A)(\bv) \to (UB \; | \; UA)(\bv)];$
\item $T(R) \vdash \forall \bv \; [(BV \; | \; A)(\bv) \to (B \; | \; A)(\bv)];$ and
\item $T(R) \vdash \forall \bv \; [(B \; | \; A)(\bv) \to (B \; | \; A + BG)(\bv)],$
\end{enumerate}
and may therefore be interpreted as rules of inference. If the relation $(B \; | \; A) \vdash_n (B' \; | \; A')$ is defined to hold in $L'_n (R)$ whenever
$$T(R) \vdash \forall \bv \; (B \; | \; A)(\bv) \to (B' \; | \; A')(\bv),$$
then we see, in view of Proposition~\ref{transitive}, that $(B \; | \; A) \leq_n (B' \; | \; A')$ implies
$(B \; | \; A) \vdash_n (B' \; | \; A').$ The following completeness theorem is due to Prest.

\begin{theorem} {\bf (Lemma Presta I)} \label{Lemma Presta 1}~\cite[Lemma 1.1.13 and Cor.\ 1.1.16]{Mike's book}  
Given $(B \; | \; A)$ and $(B' \; | \; A')$ in $L'_n (R),$ 
$$(B \; | \; A) \leq_n (B' \; | \; A') \;\;\; \mbox{if and only if} \;\;\; (B \; | \; A) \vdash_n (B' \; | \; A').$$
\end{theorem}

Theorem~\ref{Lemma Presta 1} is a completeness theorem, because it implies that any implication of the form 
$T(R) \vdash \forall \bv \; (B \; | \; A)(\bv) \to (B' \; | \; A')(\bv)$ has a proof using exclusively the rules of inference RoD (1)-(3)

\subsection{Nonhomogeneous Systems} Let $A$ be an $m \times n$ matrix; $B$ and $m \times k$ matrix and ${_R}M$ a left $R$-module. Let us consider the {\em nonhomogeneous system} of linear equations
$$A \bv \doteq \bb,$$
where $\bv$ is a column $n$-vector of variables $(v_i)$ and $\bb$ a column $k$-vector with entries from $M.$ Denote by
$$\Sol_M (A \bv \doteq \bb) := \{ \ba \in ({_R}M)^n \; : \; A\ba = \bb \}$$
the subgroup of $({_R}M)^n$ of solutions in $M$ to the nonhomogeneous system. As in the classical case over a field, if $\Sol_M (A \bv \doteq \bb)$ is nonempty, then it is a coset of the subgroup of solutions $\Sol_M (A \bv \doteq \zero)$ of the corresponding homogeneous system. Given an element $(B \; | \; A) \in L'_n (R),$ define
$$(B \; | \; A)(M) := \bigcup_{\bb \in BM^k} \Sol_M (A \bv \doteq \bb).$$
This is consistent with model-theoretic notation, because $(B \; | \; A)(M)$ is the subset of $({_R}M)^n$
defined in ${_R}M$ by the formula $(B \; | \; A)(\bv),$
\begin{eqnarray*}
(B \; | \; A)({_R}M) & = & \{ \ba \in ({_R}M)^n \; : \; \exists \bc \in ({_R}M)^k \; (B\bc = A\ba) \} \\
& = & \{ \ba \in ({_R}M)^n \; : \; M \models (B \; | \; A)(\ba ) \}.
\end{eqnarray*}
It is an exercise to show that $(B \; | \; A)(M)$ is a subgroup of $(M^n)_{\bZ},$ functorial in ${_R}M.$ Subgroups of the form are called $(B \; | \; A)(M)$ are known by several names : {\em pp-definable subgroups~\cite{Mike's book}, finite matrix subgroups~\cite{Zim}, subgroups of finite definition~\cite{GJ}, etc.}

Let
$\Latt_n (M_{\bZ})$ denote the lattice of subgroups of $(M^n)_{\bZ};$ a map
$\Ev'_M : L'_n (R) \to \Latt_n (M_{\bZ})$ is obtained by the rule $(B \; | \; A) \mapsto (B \; | \; A)({_R}M).$ By Lemma Presta~I, this a morphism of pre-orders, and, because $\Latt_n (M_{\bZ})$ is a partial order, the morphism factors through $L_n (R)$ to yield a morphism 
$$\Ev_M : L_n (R) \to \Latt_n (M_{\bZ}), \;\; [B \; | \; A] \mapsto (B \; | \; A)(M)$$
satisfying the following properties:
\begin{enumerate}
\item if $[B \; | \; A] \leq_n [B' \; | \; A']$ holds in $L_n (R),$ then 
$(B \; | \; A)({_R}M) \leq_n (B' \; | \; A')({_R}M)$ holds in $\Latt_n (M_{\bZ});$
\item $[B \; | \; A] \wedge [B' \; | \; A'] \mapsto (B \; | \; A)(M) \cap (B' \; | \; A')(M);$
\item $[B \; | \; A] +  [B' \; | \; A'] \mapsto (B \; | \; A)(M) + (B' \; | \; A')(M);$ and
\item $1_n (M) = M^n$ and $0_n (M) = 0.$
\end{enumerate} 

We may thus define the relation $(B \; | \; A) \models_n (B' \; | \; A')$ to hold in $L'_n (R)$ provided that 
$(B \; | \; A)(M) \subseteq (B' \; | \; A')(M)$ for every left $R$-module ${_R}M.$ Equivalently, 
$$T(R) \models \forall \bv \; (B|A\bv \to B'|A'\bv).$$
By G\"{o}del's Completeness Theorem, the relations $\vdash_n$ and $\models_n$ on $L'_n (R)$ are equal, so that one 
obtains a second version of Lemma Presta.\bigskip

\begin{proposition} \label{Lemma Presta 2} {\bf (Lemma Presta II.)} 
Given $(B \; | \; A)$ and $(B' \; | \; A')$ in $L'_n (R),$ 
$$(B \; | \; A) \leq_n (B' \; | \; A') \;\;\; \mbox{if and only if} \;\;\; (B \; | \; A) \models_n (B' \; | \; A').$$
\end{proposition}

\subsection{The Tensor Product and Duality} If $(B \; | \; A) \in L'_n (R),$ then, in a right $R$-module $K_R,$ its dual defines the subgroup
$$\left( \begin{array}{c|c}
B^{\tr} & \zero \\ A^{\tr} & I_n  \end{array} \right)(K_R) = 
\{ \ba \in K^n \; : \; K \models \exists \bw \; (\bw A \doteq \ba \wedge \bw B \doteq \zero) \},$$
where the action of $R^{\op}$ is written on the right. Define the relation 
$(B \; | \; A) \triangleleft_n (B' \; | \; A')$ to hold in $L'_n (R)$ provided that for every right $R$-module $K_R$ and left $R$-module ${_R}M,$ the subgroup
$$\left( \begin{array}{c|c}
(B')^{\tr} & \zero \\ (A')^{\tr} & I_n  \end{array} \right)(K) \tensor (B \; | \; A)(M) = 0$$
in the tensor product $K \tensor_R M.$ 

Let us verify that if $(B \; | \; A) \leq_n (B' \; | \; A')$ holds in $L'_n (R),$ then
$(B \; | \; A) \triangleleft_n (B' \; | \; A').$ We are given $U,$ $V$ and $G$ such that $UB = B'V$ and $UA = A' + B'G.$
Let ${\dis \ba \in \left( \begin{array}{c|c} (B')^{\tr} & \zero \\ (A')^{\tr} & I_n  \end{array} \right)(K)}$ 
and $\bb \in (B \; | \; A)(M).$ Then there exist $\bc \in K^{m'}$ such that
$\bc A' = \ba$ and $\bc B' = \zero;$ and $\bd \in M^k$ such that $B\bd = A\bb.$ Then the tensor 
$$\ba \tensor \bb = \sum_{i=1}^n \; a_i \tensor b_i$$
simplifies in $K \tensor_R M$ as
\begin{eqnarray*}
\ba \tensor \bb & = & \bc A' \tensor \bb = \bc (UA - B'G) \tensor \bb \\
& = & \bc UA \tensor \bb = \bc U \tensor A\bb \\
& = & \bc U \tensor B\bd = \bc UB \tensor \bd \\
& = & \bc B'V \tensor \bd = \zero \tensor \bd = 0.  
\end{eqnarray*}
Note how the property that $\bc B' = 0$ has been used twice in this simplification. By Corollary 1.3.8 
of~\cite{Mike's book}, the converse also holds.

\begin{proposition} \label{tensors}
Given $(B \; | \; A)$ and $(B' \; | \; A')$ in $L'_n (R),$ 
$$(B \; | \; A) \leq_n (B' \; | \; A') \;\;\; \mbox{if and only if} \;\;\; 
(B \; | \; A) \triangleleft_n (B' \; | \; A').$$
\end{proposition}

%% file: intro.bbl
\begin{thebibliography}{999}

\bibitem{Adel} Adelman, M., Abelian categories over additive ones, Journal of Pure and Applied Algebra {\bf 3} (1973), 103-117.
\bibitem{CB} Crawley-Boevey, W.W., {\em Representations of Algebras and Related Topics,} Tachikawa, H.\ and Brenner, S., eds. London Mathematical Society Lecture Note Series {\bf 168}, 127-184.
\bibitem{Jac} Jacobson, N., {\em Basic Algebra I,} 2nd ed., W.H.\ Frieman and Co., 1985.
\bibitem{Good} Goodearl, K., {\em Von Neumann Regular Rings,} Krieger Publishing Co., 1991.
\bibitem{Gou} Goursat, M.E., Sur les Substitutions Orthogonales et les Divisions R\'eguli\`eres de L'\'Espace, Annales Scientifiques de L'\'Ecole Normale Sup\'erieure {\bf 6,} Troisi\`eme S\'erie (1889), 9-102.
\bibitem{GJ} Gruson, L.\ and Jensen, C.U., Modules alg\'ebriquement compacts et foncteurs ${\dis \lim_{\leftarrow}}^i,$
C.R.\ Acad.\ Sci.\ Paris {\bf 276} (1973), 1651-1653.
\bibitem{Lick} Lickorish, W.B.R., {\em An Introduction to Knot Theory,} Graduate Texts in Mathematics {\bf 175,} Springer, 1997.
\bibitem{Mass} Massey, W.S., {\em Singular Homology Theory,} Graduate Texts in Mathematics {\bf 70,} Springer, 1980.
\bibitem{Maz} Mazet, P., Caract\'{e}rization des \'{e}pimorphismes par relation et g\'{e}n\'{e}rateurs, Seminaire P.\ Samuel {\bf 2}, 1967-68.
\bibitem{Mike's book} Prest, M., {\em Purity, Spectra and Localization,} Cambridge University Press, 2009.
\bibitem{Rosen} Rosenberg, J., {\em Algebraic $K$-Theory and Its Applications,} Graduate Texts in Mathematics {\bf 147,}
Springer, 147.
\bibitem{Scho} Schofield, A., {\em Representations of rings over skew fields,} London Mathematical Lecture Note Series {\bf 92.}
\bibitem{Sil} Silver, L., Noncommutative localization and applications, J.\ of Algebra {\bf 7} (1967), 44-76.
\bibitem{Zim} Zimmermann, W., Rein injektive direkte Summen von Moduln, Communications in Algebra {\bf 5}(1977),  1083-1117.
\bibitem{ZZH} Zimmermann-Huisgen, B.\ and Zimmermann, W., On the sparsity of representations of rings of pure global dimension zero, Transactions of the AMS {\bf 320}(1990), 695-713. 
\end{thebibliography}
